\documentclass[english]{article}
\usepackage{geometry}
\usepackage{color}
\usepackage{array}
\usepackage{float}
\usepackage{multirow}
\usepackage{amsmath}
\usepackage{amsthm}
\usepackage{amssymb}
\usepackage{graphicx}
\usepackage{esint}

\makeatletter

\providecommand{\tabularnewline}{\\}
\floatstyle{ruled}
\newfloat{algorithm}{tbp}{loa}
\providecommand{\algorithmname}{Algorithm}
\floatname{algorithm}{\protect\algorithmname}

\theoremstyle{plain}
\newtheorem{thm}{\protect\theoremname}[section]
\theoremstyle{plain}
\newtheorem{assumption}[thm]{\protect\assumptionname}
\theoremstyle{definition}
\newtheorem{defn}[thm]{\protect\definitionname}
\theoremstyle{plain}
\newtheorem{lem}[thm]{\protect\lemmaname}
\theoremstyle{remark}
\newtheorem{rem}[thm]{\protect\remarkname}

\makeatother

\usepackage{babel}
\providecommand{\assumptionname}{Assumption}
\providecommand{\definitionname}{Definition}
\providecommand{\lemmaname}{Lemma}
\providecommand{\remarkname}{Remark}
\providecommand{\theoremname}{Theorem}

\begin{document}
\title{An Inexact Primal-Dual Algorithm for Semi-Infinite Programming}
\author{Bo Wei, William B. Haskell, and Sixiang Zhao}
\maketitle
\begin{abstract}
This paper considers an inexact primal-dual algorithm for semi-infinite
programming (SIP) for which it provides general error bounds. To implement
the dual variable update, we create a new prox function for nonnegative
measures which turns out to be a generalization of the Kullback-Leibler
divergence for probability distributions. We show that under suitable
conditions on the error, this algorithm achieves an $\mathcal{O}(1/\sqrt{K})$
rate of convergence in terms of the optimality gap and constraint
violation. We then use our general error bounds to analyze the convergence
and sample complexity of a specific primal-dual SIP algorithm based
on Monte Carlo integration. Finally, we provide numerical experiments
to demonstrate the performance of our algorithm. 
\end{abstract}

\section{Introduction }

A semi-infinite programming (SIP) problem has finitely many variables
appearing in infinitely many constraints. SIP has abundant applications
such as Chebyshev approximation, robotics control, engineering design,
statistical design, and mechanical stress of materials. We refer the
reader to recent surveys in \cite{hettich1993semi,lopez2007semi}.

Primal-dual methods are a natural class of algorithms for dealing
with constrained optimization, especially when the constraints are
complex. Due to the infinitely many constraints in SIP, the SIP constraints
cannot usually be handled exactly by numerical methods. In this paper,
we develop a framework for inexact primal-dual SIP algorithms. A common
template for first-order primal-dual algorithms involves three main
steps: (i) computing the gradients of the Lagrangian with respect
to the primal and dual variables; (ii) taking a primal descent step
and a dual ascent step; and (iii) projecting the output onto any implicit
set constraints (e.g. see \cite{mahdavi2012online}). For the SIP
setting, we focus on inexact execution of these steps and develop
corresponding error bounds.

\subsection{Related Work}

We refer the reader to \cite{hettich1993semi,lopez2007semi,shapiro2009semi,bonnans2013perturbation,goberna2013semi}
for recent detailed overviews of SIP. The main computational difficulty
in SIP comes from the infinitely many constraints, and several practical
schemes have been proposed to remedy this difficulty \cite{reemtsen1998numerical,goberna2002linear,lopez2007semi,goberna2017recent}.
We offer the following very rough classification of SIP methods based
on \cite{hettich1993semi,reemtsen1998numerical,lopez2007semi}. 

\textit{Exchange methods}: In exchange methods, in each iteration
a set of new constraints is exchanged for the previous set (there
are many ways to do this). Cutting plane methods are a special case
where constraints are never dropped. The algorithm in \cite{gribik1979central}
is the prototype for several SIP cutting plane schemes, and it has
been improved in various ways \cite{kortanek1993central,betro2004accelerated,mehrotra2014cutting}.
In particular, a new exchange method is proposed in \cite{zhang2010new}
that only keeps those active constraints with positive Lagrange multipliers.
New constraints are selected using a certain computationally-cheap
criterion. In \cite{mehrotra2014cutting}, the earlier central cutting
plane algorithm from \cite{kortanek1993central} is extended to allow
for nonlinear convex cuts.

Randomized cutting plane algorithms have recently been developed for
SIP in \cite{Calafiore_Uncertain_2005,campi2008exact,esfahani2015performance}.
The idea is to input a probability distribution over the constraints,
randomly sample a modest number of constraints, and then solve the
resulting relaxed problem. Intuitively, as long as a sufficient number
of samples of the constraints is drawn, the resulting randomized solution
should violate only a small portion of the constraints and achieve
near-optimality.

\textit{Discretization methods}: In the discretization approach,
a sequence of relaxed problems with a finite number of constraints
is solved according to a predefined or adaptively controlled grid
generation scheme \cite{reemtsen1991discretization,still2001discretization}.
Discretization methods are generally computationally expensive. The
convergence rate of the error between the solution of the SIP problem
and the solution of the discretized program is investigated in \cite{still2001discretization}. 

\textit{Local reduction methods}: In the local reduction approach,
an SIP problem is reduced to a problem with a finite number of constraints
\cite{gramlich1995local}. The reduced problem involves constraints
which are defined only implicitly, and the resulting problem is solved
via the Newton method which has good local convergence properties.
However, local reduction methods require strong assumptions and are
often conceptual.

\textit{Dual methods}: A wide class of SIP algorithms is based on
directly solving the KKT conditions. In \cite{ito2000dual,liu2002adaptive,liu2004new},
the authors derive Wolfe's dual for an SIP and discuss numerical schemes
for this problem. The KKT conditions often have some degree of smoothness,
and so various Newton-type methods can be applied \cite{qi2003semismooth,li2004smoothing,ni2006truncated,qi2009smoothing}.
However, feasibility is not guaranteed under the all Newton-type methods.
A new smoothing Newton-type method is proposed to overcome this drawback
in \cite{ling2010new}.

\textit{Applications}: SIP is the basis of the approximate linear
programming (ALP) approach for dynamic programming. In \cite{deFarias_Sampling_2004,bhat2012non,esfahani2017infinite},
the authors consider various schemes for solving ALPs. The idea of
random constraint sampling in ALP is developed in \cite{deFarias_Sampling_2004,bhat2012non}.
In \cite{lin2017revisiting}, an adaptive constraint sampling approach
called 'ALP-Secant' is proposed which is based on solving a sequence
of saddle-point problems. It is shown that ALP-Secant returns a near
optimal ALP solution and a lower bound on the optimal cost with high
probability in a finite number of iterations.

Many risk-aware optimization models also depend on SIP (e.g. \cite{noyan2013optimization,noyan2018optimization}),
in particular, risk-constrained optimization (e.g. \cite{dentcheva2003optimization,dentcheva2004optimality,dentcheva2009optimization,dentcheva2015optimization,haskell2017primal,homem2009cutting,hu2012sample}).
In \cite{dentcheva2003optimization,dentcheva2004optimality,dentcheva2009optimization,haskell2013optimization},
a duality theory for stochastic dominance constrained optimization
is developed which shows the special role of utility functions as
Lagrange multipliers. Relaxations of multivariate stochastic dominance
have been proposed based on various parametrized families of utility
functions, see \cite{dentcheva2009optimization,homem2009cutting,hu2012sample,haskell2013optimization}.
Computational aspects of the increasing concave stochastic dominance
constrained optimization are discussed in \cite{homem2009cutting,hu2012sample,haskell2017primal}.

\subsection{Contributions}

We make the following contributions in the present work:
\begin{enumerate}
\item Since the dual variables are nonnegative measures on the constraint
index set in SIP, we create a new prox function for nonnegative measures
which generalizes the Kullback-Leibler divergence for probability
distributions. This new prox function is essential for the dual variable
update in our algorithm. In particular, it allows us to obtain explicit
analytical solutions in the dual variable update.
\item We propose a primal-dual algorithm for SIP that simultaneously optimizes
the primal and dual variables. It is computationally difficult to
implement this algorithm exactly due to the infinite number of constraints,
so we conduct a general error analysis of inexact implementation.
In particular, we show that under suitable conditions on the errors,
an inexact algorithm achieves an $\mathcal{O}(1/\sqrt{K})$ rate of
convergence in terms of the optimality gap and constraint violation.
\item We propose a specific primal-dual SIP algorithm based on Monte Carlo
integration to approximate the high dimensional integrals in the primal
update. We show that as long as the number of samples is large enough,
the Monte Carlo integration can be made into an arbitrarily good approximation
of the high dimensional integrals. Based on our general error bounds,
we show that this specific algorithm achieves an $\mathcal{O}(1/\sqrt{K})$
rate of convergence in probability.
\end{enumerate}
This paper is organized as follows. In Section 2, we review preliminary
material on convex SIP and its duality theory. In Section \ref{sec:The-Regularized-Saddle-Point}
we introduce our new prox function for nonnegative measures, and construct
a regularized saddle-point problem which closely approximates our
original SIP. The following Section \ref{sec:The-Primal-Dual-Algorithm}
analyzes an inexact primal-dual algorithm for the regularized saddle
point problem, and then reports our error analysis. As a specific
example, Monte Carlo integration is used in our inexact primal-dual
algorithm in Section \ref{sec:MC_integration}. All of our main results
are proved within Sections \ref{sec:The-Regularized-Saddle-Point},
\ref{sec:The-Primal-Dual-Algorithm}, and \ref{sec:MC_integration},
and all supporting results are gathered together in Section \ref{sec:Proofs-of-Main Results}
for better readability. We then present some numerical experiments
for the inexact primal-dual algorithm using Monte Carlo integration
in Section \ref{sec:Numerical-Experiments}. Finally, we conclude
the paper in Section \ref{sec:Conclusion}.

\section{\label{sec:Preliminaries}Preliminaries}

We begin with the following ingredients:
\begin{description}
\item [{A1}] Convex, compact decision set $\mathcal{X}\subset\mathbb{R}^{m}$;
\item [{A2}] Convex objective function $f\text{ : }\mathcal{X}\rightarrow\mathbb{R}$,
which is $L_{f}-$Lipschitz continuous and differentiable;
\item [{A3}] Compact, full-dimensional, convex constraint index set $\Xi\subset\mathbb{R}^{d}$;
\item [{A4}] Constraint function $g\text{ : }\mathcal{X}\times\Xi\rightarrow\mathbb{R}$,
such that for all $\xi\in\Xi$, $x\rightarrow g\left(x,\,\xi\right)$
is convex, $L_{g,\mathcal{X}}-$Lipschitz continuous (uniformly in
$\xi$), and continuously differentiable;
\item [{A5}] For all $x\in\mathcal{X}$, $\xi\rightarrow g\left(x,\,\xi\right)$
is $L_{g,\,\Xi}-$Lipschitz continuous (uniformly in $x$).
\end{description}
Define $\mathcal{C}\left(\Xi\right)$ to be the space of continuous
functions in the supremum norm $\left\Vert f\right\Vert _{\mathcal{C}\left(\Xi\right)}\triangleq\sup_{\xi\in\Xi}\left|f(\xi)\right|$.
The cone of all non-negative continuous functions in $\mathcal{C}\left(\Xi\right)$
is denoted $\mathcal{C}_{+}\left(\Xi\right)$, and for any $f_{1},f_{2}\in\mathcal{C}\left(\Xi\right)$,
$f_{1}\leq f_{2}$ means $f_{2}-f_{1}\in\mathcal{C}_{+}\left(\Xi\right)$. 

We may define the mapping $G\text{ : }\mathcal{X}\rightarrow\mathcal{C}\left(\Xi\right)$
where $\left[G\left(x\right)\right]\left(\xi\right)=g\left(x,\,\xi\right)$
for all $\xi\in\Xi$ to succinctly capture all the constraints. The
constraints $G\left(x\right)$ lie in $\mathcal{C}\left(\Xi\right)$,
which is a Banach space. Our (conic form) SIP is then
\[
\mathbb{P}:\inf_{x\in\mathcal{X}}\left\{ f\left(x\right)\mbox{ : }G\left(x\right)\leq0\right\} \equiv\inf_{x\in\mathcal{X}}\left\{ f\left(x\right)\mbox{ : }g\left(x,\,\xi\right)\leq0,\,\forall\xi\in\Xi\right\} .
\]
Problem $\mathbb{P}$ is a convex optimization problem under Assumptions
\textbf{A1}, \textbf{A2}, and \textbf{A4}. We assume that Problem
$\mathbb{P}$ is solvable and that the Slater constraint qualification
holds.
\begin{assumption}
\label{SolvabilityandSlater}(i) (Solvability) An optimal solution
$x^{\ast}$ of Problem \textup{$\mathbb{P}$ exists. }

(ii) (Slater condition) There exists an $\tilde{x}\in\mathcal{X}$
such that $g\left(\tilde{x},\,\xi\right)<0$ for all $\xi\in\Xi$. 
\end{assumption}

Now we recall the Lagrangian dual of Problem $\mathbb{P}$. Let $\mathcal{B}$
be the Borel sigma algebra on $\Xi$ so that $\left(\Omega,\,\Xi\right)$
is a measurable space. The Lagrange multipliers of the conic constraint
$G\left(x\right)\leq0$ will lie in the dual to $\mathcal{C}\left(\Xi\right)$,
which is $\mathcal{C}\left(\Xi\right)^{\ast}=\mathcal{M}\left(\Xi\right)$,
the space of finite signed measures on $(\Xi,\,\mathcal{B})$ equipped
with the total variation norm $\left\Vert \varLambda\right\Vert _{TV}\triangleq\intop_{\Xi}\left|\varLambda\right|(d\xi)$.
For later use, we let $\delta_{\xi}\in\mathcal{M}_{+}\left(\Xi\right)$
denote the point measure at $\xi\in\Xi$.

The duality pairing between $\varLambda\in\mathcal{M}\left(\Xi\right)$
and $f\in\mathcal{C}\left(\Xi\right)$ is given by $\left\langle \varLambda,\,f\right\rangle \triangleq\intop_{\Xi}f(\xi)\varLambda(d\xi)$.
The Lagrangian for Problem $\mathbb{P}$ is $L\text{ : }\mathcal{X}\times\mathcal{M}\left(\Xi\right)\rightarrow\mathbb{R}$
defined by $L\left(x,\,\varLambda\right)\triangleq f\left(x\right)+\left\langle \varLambda,\,G(x)\right\rangle $.
Let $\mathcal{M}_{+}\left(\Xi\right)$ denote the set of non-negative
measures in $\mathcal{M}\left(\Xi\right)$, then we may write the
saddle-point form of Problem $\mathbb{P}$ as 

\[
\mathbb{SP}:\inf_{x\in\mathcal{X}}\sup_{\varLambda\in\mathcal{M}_{+}\left(\Xi\right)}\{f\left(x\right)+\left\langle \varLambda,\,G(x)\right\rangle \}.
\]
A pair $(\widetilde{x},\,\widetilde{\varLambda})\in\mathcal{X}\times\mathcal{M}_{+}\left(\Xi\right)$
is said to be a saddle-point (of Problem $\mathbb{SP}$) if 

\[
L(\widetilde{x},\,\varLambda)\leq L(\widetilde{x},\,\widetilde{\varLambda})\leq L(x,\,\widetilde{\varLambda}),\:\forall(x,\,\varLambda)\in\mathcal{X}\times\mathcal{M}_{+}\left(\Xi\right).
\]

Problem $\mathbb{SP}$ is equivalent to Problem $\mathbb{P}$, since
it assigns infinite cost to any infeasible $x\in\mathcal{X}$ that
does not satisfy $G\left(x\right)\leq0$. The dual to Problem $\mathbb{P}$
is then obtained by interchanging the order of inf and sup in Problem
$\mathbb{SP}$ to obtain

\[
\mathbb{D}:\sup_{\varLambda\in\mathcal{M}_{+}\left(\Xi\right)}\inf_{x\in\mathcal{X}}L\left(x,\,\varLambda\right)\equiv\sup_{\varLambda\in\mathcal{M}_{+}\left(\Xi\right)}d\left(\varLambda\right),
\]
where $d\left(\varLambda\right)\triangleq\inf_{x\in\mathcal{X}}L\left(x,\,\varLambda\right)$
is the dual functional (which is always concave as the infimum of
linear functions). For later use, we let $\text{val}(\mathbb{D})$
denote the optimal value of Problem $\mathbb{D}$. The next result
shows that both optimal sets of the primal and dual Problems $\mathbb{P}$
and $\mathbb{D}$ are non-empty and there is no duality gap under
our assumptions.
\begin{thm}
\label{SP}Suppose Assumption \ref{SolvabilityandSlater} holds.

(i) Let $x^{*}$ be an optimal solution of Problem $\mathbb{P}$,
then there exists a solution $\varLambda^{\ast}\in\mathcal{M}_{+}\left(\Xi\right)$
to Problem $\mathbb{D}$ such that $(x^{\ast},\,\varLambda^{\ast})$
is a saddle point for Problem $\mathbb{SP}$\textup{,} i.e.,

\[
L(x^{\ast},\,\varLambda^{\ast})=\min_{x\in\mathcal{X}}\max_{\varLambda\in\mathcal{M}_{+}\left(\Xi\right)}L(x,\,\varLambda)=\max_{\varLambda\in\mathcal{M}_{+}\left(\Xi\right)}\min_{x\in\mathcal{X}}L(x,\,\varLambda).
\]

(ii) Let $\left(x^{*},\,\varLambda^{*}\right)\in\mathcal{X}\times\mathcal{M}_{+}\left(\Xi\right)$
be a saddle-point of Problem $\mathbb{SP}$, then $x^{*}$ and $\varLambda^{*}$
are primal and dual optimal, respectively.
\end{thm}

Theorem \ref{SP} demonstrates the existence of a saddle point for
Problem $\mathbb{SP}$ under Assumption \ref{SolvabilityandSlater},
and reveals the relationship between the saddle-points of Problem
$\mathbb{SP}$ and the solutions of Problems $\mathbb{P}$ and $\mathbb{D}$.

~
\begin{center}
\begin{tabular}{l}
\hline 
Key notation\tabularnewline
\hline 
$\|\cdot\|$ is the Euclidean norm\tabularnewline
$B_{r}\left(x\right)$ is the Euclidean ball with radius $r$ centered
at $x$ \tabularnewline
$\text{vol}\left(\cdot\right)$ is the volume of a set\tabularnewline
$\Gamma(\cdot)$ is the gamma function\tabularnewline
$D\left(\cdot,\,\cdot\right)$ is the Kullback-Leibler divergence\tabularnewline
$\ll$ denotes absolute continuity of measures\tabularnewline
$\text{val}\left(\cdot\right)$ is the optimal value of an optimization
problem\tabularnewline
$\text{sol}\left(\cdot\right)$ is the set of solutions of an optimization
problem\tabularnewline
$D_{\mathcal{X}},\,D_{\Xi}$ is the diameter of $\mathcal{X},\,\Xi$,
respectively\tabularnewline
$G_{\max}$ is an upper bound on $\left\Vert G(x)\right\Vert _{\mathcal{C}\left(\Xi\right)}$
for all $x\in\mathcal{X}$\tabularnewline
\hline 
\end{tabular}
\par\end{center}

\section{\label{sec:The-Regularized-Saddle-Point}The Regularized Saddle-Point
Problem}

In this section, we construct a \textit{regularized} version of Problem
$\mathbb{SP}$ that is more amenable to a primal-dual algorithm. First,
we show that we can restrict the dual feasible region of Problem $\mathbb{D}$
to within a bounded set under the Slater condition (Subsection \ref{subsec:Dual-Bound}).
Second, we introduce a new prox-function for $\mathcal{M}_{+}\left(\Xi\right)$
so that we may do mirror descent type updates for the dual variables
(Subsection \ref{subsec:Distance-metric-between}). Finally, we introduce
the corresponding regularized saddle-point problem (Subsection \ref{subsec:Relationship-between-Original}).

\subsection{\label{subsec:Dual-Bound}Dual Bound}

We first examine the set of dual optimal solutions

\[
\text{sol}\left(\mathbb{D}\right)\triangleq\left\{ \varLambda\in\mathcal{M}_{+}\left(\Xi\right):\,d\left(\varLambda\right)\geq\text{val}(\mathbb{D})\right\} .
\]
Under Assumption \ref{SolvabilityandSlater}, there is an $\tilde{x}\in\mathcal{X}$
that satisfies $\alpha\triangleq\min_{\xi\in\Xi}\left\{ -g(\tilde{x},\xi)\right\} >0$
(this expression is well defined under Assumptions \textbf{\textcolor{black}{A3}}
and \textbf{\textcolor{black}{A5}}). The following Theorem \ref{dualupperbound}
shows that we can restrict the set of optimal solutions of Problem
$\mathbb{D}$ to within a bounded set (in the total variation norm).
\begin{thm}
\label{dualupperbound} For all $\varLambda^{*}\in\mathrm{sol}\left(\mathbb{D}\right)$,
$\left\Vert \text{\ensuremath{\varLambda^{\ast}}}\right\Vert _{TV}\leq(f(\tilde{x})-\mathrm{val}(\mathbb{D}))/\alpha$.
\end{thm}

The previous result is a key component of our upcoming regularized
saddle-point problem.

\subsection{\label{subsec:Distance-metric-between}Prox Function for Nonnegative
Measures}

First, we define
\[
\mathcal{P}\left(\Xi\right)\triangleq\{\phi\in\mathcal{M}_{+}\left(\Xi\right):\,\int_{\Xi}\phi(d\xi)=1\}
\]
to be the probability simplex on $\Xi$. In this subsection we present
our new prox function for $\mathcal{M}_{+}\left(\Xi\right)$ which
generalizes the Kullback-Leibler (KL) divergence

\[
D\left(\phi,\varphi\right)\triangleq\mathbb{E}_{\widetilde{\xi}\sim\phi}\left[\log\left(\frac{\phi(\widetilde{\xi})}{\varphi(\widetilde{\xi})}\right)\right]=\int_{\Xi}\log(\frac{\phi(\xi)}{\varphi(\xi)})\phi(d\xi)
\]
for probability densities $\phi,\varphi\in\mathcal{P}\left(\Xi\right)$.
Our new prox function for $\mathcal{M}_{+}\left(\Xi\right)$ is defined
as follows.
\begin{defn}
\label{defofB}For any $\varLambda,\,\varGamma\in\mathcal{M}_{+}\left(\Xi\right)$,

\[
B(\varLambda,\,\varGamma)\triangleq\begin{cases}
\intop_{\Xi}\log(\frac{\varLambda(\xi)}{\varGamma(\xi)})\varLambda(d\xi)-\left(\intop_{\Xi}\varLambda(d\xi)-\intop_{\Xi}\varGamma(d\xi)\right), & \varLambda\ll\varGamma\text{ and }\intop_{\Xi}\left|\log(\frac{\varLambda(\xi)}{\varGamma(\xi)})\right|\varLambda(d\xi)<\infty,\\
+\infty, & \text{otherwise.}
\end{cases}
\]
\end{defn}

In addition, we define the key mappings $\rho\text{ : }\mathcal{M}_{+}\left(\Xi\right)\rightarrow\mathbb{R}$
and $\phi\text{ : }\mathcal{M}_{+}\left(\Xi\right)\rightarrow\mathcal{P}\left(\Xi\right)$
via

\[
\begin{array}{rcl}
\rho(\varLambda)\triangleq\left\Vert \varLambda\right\Vert _{TV}\in\mathbb{R}_{+} & \text{and} & \phi(\varLambda)\triangleq\varLambda/\left\Vert \varLambda\right\Vert _{TV}\in\mathcal{P}(\Xi),\end{array}
\]
respectively. The mapping $\rho$ is shorthand for the total variation
norm of a measure, and the mapping $\phi$ is the projection of $\mathcal{M}_{+}\left(\Xi\right)$
onto $\mathcal{P}\left(\Xi\right)$. The function $H:\,\mathbb{R}_{+}\times\mathbb{R}_{++}\rightarrow\mathbb{R}$,
defined via $H(\rho,\,\rho')\triangleq\rho\log(\rho/\rho')-\rho+\rho'$,
also plays an important role in our analysis for connecting our new
concepts and definitions with the KL divergence $D$.
\begin{lem}
\textup{\label{lem:prox_preliminary}(i) }There is a one-to-one correspondence
between measures $\varLambda\in\mathcal{M}_{+}\left(\Xi\right)$ and
pairs $(\rho(\varLambda),\,\phi(\varLambda))\in\mathbb{R}_{+}\times\mathcal{P}(\Xi)$.

(ii) The mapping $\rho\rightarrow H(\rho,\,\rho')$ is convex in $\rho>0$,
for all fixed $\rho'>0$. Moreover, $H(\rho,\,\rho')\geq0$ for all
$\rho>0$ and $\rho'>0$, and $H(\rho,\,\rho')=0$ if and only if
$\rho=\rho'$.
\end{lem}

We report the main properties of $B$ in the following theorem.
\begin{thm}
\label{propertiesofB}(i) For any $\varLambda,\,\Gamma\in\mathcal{M}_{+}\left(\Xi\right)$,

\[
B(\varLambda,\,\varGamma)=\rho(\varLambda)\,D(\phi(\varLambda),\,\phi(\varGamma))+H(\rho(\varLambda),\,\rho(\varGamma)).
\]

(ii) For any two non-negative measures $\varLambda,\,\varGamma\in\mathcal{M}_{+}\left(\Xi\right)$,
$B(\varLambda,\,\varGamma)$ is non-negative. Further, $B(\varLambda,\,\varGamma)=0$
if and only if $\varLambda=\varGamma$, i.e., $\rho(\varLambda)=\rho(\varGamma)$
and $\phi(\varLambda)=\phi(\varGamma)$. 

(iii) If $\phi,\varphi\in\mathcal{P}\left(\Xi\right)$\textup{ then
$B(\phi,\,\varphi)=D(\phi,\,\varphi)$.}

(iv) The mapping $\varLambda\rightarrow B(\varLambda,\,\varGamma)$
is convex for any fixed $\varGamma\in\mathcal{M}_{+}\left(\Xi\right)$.

(v) For any two non-negative measures $\varLambda,\varGamma\in\mathcal{M}_{+}(\Xi)$
with $\left\Vert \varLambda\right\Vert _{TV}\leq\rho$ and $\left\Vert \varGamma\right\Vert _{TV}\leq\rho$,
where $\rho>0$, we have $B(\varLambda,\,\varGamma)\geq\frac{1}{2\rho}\left\Vert \varLambda-\varGamma\right\Vert _{TV}^{2}$.
\end{thm}

\begin{rem}
Pinsker's inequality (see \cite{pinsker1960information}) states that
$D(\phi,\,\varphi)\geq\frac{1}{2}\left\Vert \phi-\varphi\right\Vert _{TV}^{2}$
for any two probability measures $\phi,\,\varphi\in\mathcal{P}(\Xi)$.
We interpret Theorem \ref{propertiesofB}(v) as a generalization of
Pinsker's inequality. 
\end{rem}

\subsection{\label{subsec:Relationship-between-Original}The Regularized Saddle-Point
Problem}

In preparation for our primal-dual algorithm, we must modify Problem
$\mathbb{SP}$ to satisfy the following requirements:
\begin{enumerate}
\item The dual feasible region must be bounded.
\item The inner maximization over $\varLambda$ must not return a point
measure, otherwise, the prox term which depends on $B$ would be unbounded
and lead to major technical difficulties.
\item The prox term $B$ must remain bounded for all iterates.
\end{enumerate}
To proceed, recalling $\alpha=\min_{\xi\in\Xi}\left\{ -g(\tilde{x},\xi)\right\} >0$,
we first define $\bar{\rho}(\theta)\triangleq\frac{1}{\alpha}(f(\tilde{x})-\text{val}(\mathbb{D}))+\theta$
for $\theta>0$. We know that all dual optimal solutions are contained
in
\[
\mathcal{V}\triangleq\{\varLambda\in\mathcal{M}_{+}\left(\Xi\right):\:\|\varLambda\|_{TV}\leq\bar{\rho}(\theta)\},
\]
by Theorem \ref{dualupperbound}. Control of the constant $\theta>0$
will appear later in our error bounds.

For some constant $\rho_{0}\in(0,\,\bar{\rho}(\theta)]$, we let $\varLambda_{u}\in\mathcal{M}_{+}\left(\Xi\right)$
be a uniform measure on $\Xi$ such that $\varLambda_{u}(\xi)=\rho_{0}\text{vol}\left(\Xi\right)^{-1}$
for all $\xi\in\Xi$. We may then define the regularized Lagrangian

\[
L_{\kappa}(x,\,\varLambda)\triangleq f\left(x\right)+\left\langle \varLambda,\,G(x)\right\rangle -\kappa\,B\left(\varLambda,\,\varLambda_{u}\right),
\]
where $\kappa\in(0,1]$ is the regularization parameter. Clearly,
$x\rightarrow L_{\kappa}(x,\,\varLambda)$ is convex and $\varLambda\rightarrow L_{\kappa}(x,\,\varLambda)$
is concave due to convexity of $\varLambda\rightarrow B\left(\varLambda,\,\varLambda_{u}\right)$
by Theorem \ref{propertiesofB}(iv). Our resulting regularized saddle-point
problem is
\[
\mathbb{SP}_{\kappa}:\inf_{x\in\mathcal{X}}\sup_{\varLambda\in\mathcal{V}}L_{\kappa}(x,\,\varLambda).
\]

The existence of an inner maximizer in Problem $\mathbb{SP}_{\kappa}$
is guaranteed by Lemma \ref{Hausdorff_compact_uppersemicont}(ii)
(which claims that the set $\mathcal{V}$ is compact in the weak-star
topology) and Lemma \ref{Hausdorff_compact_uppersemicont}(iii) (which
says that the mapping $\varLambda\rightarrow\left\langle \varLambda,\,G(x)\right\rangle -\kappa\,B\left(\varLambda,\,\varLambda_{u}\right)$
is upper semi-continuous with respect to the weak-star topology in
$\mathcal{V}$ for each $x\in\mathcal{X}$). The uniqueness of the
inner maximizer in Problem $\mathbb{SP}_{\kappa}$ is guaranteed from
Lemma \ref{lem:prox_preliminary}, Lemma \ref{Vrhophi} (which verifies
that the inner maximization of Problem $\mathbb{SP}_{\kappa}$ can
be transformed into a two-stage optimization problem) and Lemma \ref{maxoverphi}
(which provides the scaled probability measure of the inner maximization
of Problem $\mathbb{SP}_{\kappa}$). We define $\varLambda_{\kappa}\left(x\right)\in\mathcal{V}$
to be the inner maximizer in Problem $\mathbb{SP}_{\kappa}$ for fixed
$x\in\mathcal{X}$, i.e.,

\begin{equation}
\varLambda_{\kappa}(x)\in\arg\max_{\varLambda\in\mathcal{V}}\left\langle \varLambda,\,G(x)\right\rangle -\kappa\,B\left(\varLambda,\,\varLambda_{u}\right).\label{ReguSoluLambdaRx}
\end{equation}

The existence of a saddle-point $(x_{\kappa}^{\ast},\,\varLambda_{\kappa}^{\ast})$
of Problem $\mathbb{SP}_{\kappa}$ is guaranteed by Fan's minimax
theorem \cite[Theorem 1]{fan1953minimax}.
\begin{thm}
\label{existenceofSP_R} There exists a saddle point $(x_{\kappa}^{\ast},\,\varLambda_{\kappa}^{\ast})\in\mathcal{X}\times\mathcal{V}$
for Problem \textup{$\mathbb{SP}_{\kappa}$,} i.e., 

\[
\mathrm{val}(\mathbb{SP}_{\kappa})=L_{\kappa}(x_{\kappa}^{\ast},\,\varLambda_{\kappa}^{\ast})=\min_{x\in\mathcal{X}}\max_{\varLambda\in\mathcal{V}}L_{\kappa}(x,\,\varLambda)=\max_{\varLambda\in\mathcal{V}}\min_{x\in\mathcal{X}}L_{\kappa}(x,\,\varLambda).
\]
\end{thm}

We close this section by discussing the choice of the regularization
parameter $\kappa$. In particular, we want to choose $\kappa$ so
that Problem $\mathbb{SP}_{\kappa}$ is ``close'' to Problem $\mathbb{SP}$
in some sense. For this purpose, we introduce the following additional
constants:
\begin{itemize}
\item the radius $R_{\Xi}$ of the largest ball contained in $\Xi$ centered
at $\tilde{\xi}\in\Xi$ (which exists since $\Xi$ is full-dimensional);
\item the ratio
\[
r\triangleq\text{vol}\left(B_{R_{\Xi}}(\tilde{\xi})\right)/\text{vol}\left(\Xi\right)\equiv\frac{\pi^{d/2}}{\Gamma(d/2+1)}R_{\Xi}^{d}/\text{vol}\left(\Xi\right)
\]
between the volume of this ball and $\Xi$ itself;
\item the upper bound $H_{\max}\triangleq\max_{\rho\in[0,\,\bar{\rho}(\theta)]}H(\rho,\,\rho_{0})$;
\item the constant $\bar{C}\left(\theta\right)\triangleq\bar{\rho}(\theta)L_{g,\,\Xi}(R_{\Xi}+D_{\Xi})-\bar{\rho}(\theta)\log\left(r\right)+H_{\max}$.
\end{itemize}
Then, for any $\epsilon>0$, we choose the regularization parameter

\[
\bar{\kappa}(\epsilon)\triangleq\min\left\{ \frac{\epsilon}{2\,\bar{C}\left(\theta\right)},\,\left(\frac{\epsilon}{2\,\bar{\rho}(\theta)d}\right)^{2},\,\frac{\epsilon}{\rho_{0}},\,1\right\} .
\]
The following result demonstrates that, under this choice of $\bar{\kappa}\left(\epsilon\right)$,
the gap between the values of the inner maximization over $\mathcal{V}$
of the regularized Lagrangian and the original Lagrangian can be made
arbitrarily small through our control of $\epsilon$. It follows that
Problem $\mathbb{SP}_{\kappa}$ can be made into an arbitrarily good
approximation of Problem $\mathbb{SP}$.
\begin{thm}
\label{maxsupdiff_less than_epsilon} For any $x\in\mathcal{X}$ and
$\epsilon>0$, 

\[
\max_{\varLambda\in\mathcal{V}}\left\{ \left\langle \varLambda,\,G(x)\right\rangle -\bar{\kappa}(\epsilon)\,B\left(\varLambda,\,\varLambda_{u}\right)\right\} \geq\max_{\varLambda\in\mathcal{V}}\left\{ \left\langle \varLambda,\,G(x)\right\rangle \right\} -\epsilon.
\]
\end{thm}

\begin{proof}
From Lemma \ref{maxsupdiff} (which gives an explicit bound on the
gap between the values of the inner maximization of the regularized
Lagrangian and the original Lagrangian), it is sufficient to show
that the gap bound $\bar{\kappa}(\epsilon)\bar{C}\left(\theta\right)-\bar{\rho}(\theta)\bar{\kappa}(\epsilon)\log(\bar{\kappa}(\epsilon))d$
can be made arbitrarily small. In fact, from the definition of $\bar{\kappa}(\epsilon)$,
we have 

\[
\bar{\rho}(\theta)\bar{\kappa}(\epsilon)\log(\bar{\kappa}(\epsilon))d-\bar{\kappa}(\epsilon)\bar{C}\left(\theta\right)\geq\bar{\rho}(\theta)\bar{\kappa}(\epsilon)\log(\bar{\kappa}(\epsilon))d-\frac{\epsilon}{2}\geq-\bar{\rho}(\theta)\sqrt{\bar{\kappa}(\epsilon)}d-\frac{\epsilon}{2}\geq-\epsilon,
\]
where the first inequality holds since $\bar{\kappa}(\epsilon)\leq\frac{\epsilon}{2\bar{C}\left(\theta\right)}$,
the second holds because $\log(\bar{\kappa}(\epsilon))\geq-1/\sqrt{\bar{\kappa}(\epsilon)}$,
and the last one follows from $\bar{\kappa}(\epsilon)\leq\left(\frac{\epsilon}{2\bar{\rho}(\theta)d}\right)^{2}$.
\end{proof}

\section{\label{sec:The-Primal-Dual-Algorithm}The Inexact Primal-Dual Algorithm}

In this section, we consider an inexact primal-dual algorithm for
Problem $\mathbb{SP}_{\kappa}$, and this algorithm is ultimately
used to recover a solution to Problem $\mathbb{P}$. In particular,
we will focus on inexact primal updates since the dual update has
a closed form solution.

Given primal iterate $x_{k}$ at iteration $k\geq0$, the conceptual
step for the updated $x_{k+1}$ at iteration $k+1$ is
\begin{equation}
\arg\min_{x\in\mathcal{X}}\left[\gamma\left\langle x-x_{k},\,\nabla f(x_{k})+\left\langle \varLambda_{k},\,\nabla G\left(x_{k}\right)\right\rangle \right\rangle +\frac{1}{2}\|x-x_{k}\|^{2}\right],\label{primalupdate}
\end{equation}
where $\gamma>0$ is the step length. The first term in square brackets
in (\ref{primalupdate}) focuses on improvement in the objective value.
It discourages moving in directions that are not perpendicular to
$\nabla_{x}L\left(x_{k},\,\varLambda_{k}\right)=\nabla f(x_{k})+\left\langle \varLambda_{k},\,\nabla G\left(x_{k}\right)\right\rangle $.
The second term (Euclidean prox-function) stresses feasibility of
$x_{k+1}$ and discourages movement too far away from the current
$x_{k}.$ Furthermore, the second term is strongly convex which ensures
that (\ref{primalupdate}) has a unique solution. 

The main difficulty in Problem (\ref{primalupdate}) is that the term
$\left\langle \varLambda_{k},\,\nabla G\left(x_{k}\right)\right\rangle $
cannot be evaluated exactly. There are two difficulties here: (i)
it is a high dimensional integral, and (ii) we cannot compute the
gradient $\nabla G\left(x_{k}\right)$ since it involves computing
$\nabla g\left(x_{k},\xi\right)$ for all $\xi\in\Xi$. Instead, we
approximate it with some function $\mathcal{G}_{k}(x_{k},\varLambda_{k})\approx\left\langle \varLambda_{k},\,\nabla G\left(x_{k}\right)\right\rangle $
and denote the approximation error by

\[
\varepsilon_{k}\triangleq\mathcal{G}_{k}(x_{k},\varLambda_{k})-\left\langle \varLambda_{k},\,\nabla G\left(x_{k}\right)\right\rangle ,\quad\forall k\geq0.
\]
These approximation errors will appear in our final bounds on the
optimality gap and constraint violation. The corresponding inexact
primal update is then:

\begin{equation}
x_{k+1}=\arg\min_{x\in\mathcal{X}}\left[\gamma\langle x-x_{k},\,\nabla f(x_{k})+\mathcal{G}_{k}(x_{k},\varLambda_{k})\rangle+\frac{1}{2}\|x-x_{k}\|^{2}\right].\label{inexactprimalupdate}
\end{equation}

The dual update for $k\geq0$ is
\begin{equation}
\varLambda_{k+1}\in\arg\min_{\varLambda\in\mathcal{V}}\left[\gamma\left(\left\langle \varLambda-\varLambda_{k},\,-G\left(x_{k}\right)\right\rangle +\kappa\,B\left(\varLambda,\,\varLambda_{u}\right)\right)+B\left(\varLambda,\,\varLambda_{k}\right)\right],\label{dualupdate}
\end{equation}
where the regularization term $B\left(\varLambda,\,\varLambda_{u}\right)$
guarantees that the dual update will not be a point measure since
point measures are not absolutely continuous with respect to $\varLambda_{u}$,
and the third term $B\left(\varLambda,\,\varLambda_{k}\right)$ discourages
movement away from the current iterate $\varLambda_{k}$. Problem
(\ref{dualupdate}) is a convex optimization problem due to the convexity
of $\varLambda\rightarrow B(\varLambda,\,\varGamma)$ by Theorem \ref{propertiesofB}(iv).

The implementation details of our inexact primal-dual algorithm are
given in Algorithm \ref{alg:The-Inexact-Primal-Dual}. To define the
step length $\gamma$ in the primal and dual updates, we introduce
the following constants:
\begin{itemize}
\item $C'(\epsilon,\,\theta)\triangleq-\bar{\rho}(\theta)\log(\kappa(\epsilon))d+\bar{\rho}(\theta)L_{g,\,\Xi}(R_{\Xi}+D_{\Xi})-\bar{\rho}(\theta)\log\left(r\right)+H_{\max}$;
\item $C(\epsilon,\,\theta)\triangleq\max\left\{ \rho_{0},\,C'(\epsilon,\,\theta)\right\} $.
\end{itemize}
\begin{algorithm}
\caption{\label{alg:The-Inexact-Primal-Dual}The Inexact Primal-Dual Algorithm}

\textbf{\textcolor{black}{Input: }}Number of iterations $K\geq1$,
initial points $x_{0}\in\mathcal{X},$ $\varLambda_{0}=\varLambda_{u}$$,$
and constant step length

\[
\gamma=\sqrt{\frac{2(C(\epsilon,\,\theta)+D_{\mathcal{X}})}{K(\bar{\rho}(\theta)G_{\max}^{2}+2(L_{f}+\bar{\rho}(\theta)L_{g,\,\mathcal{X}})^{2})}}.
\]

\textbf{\textcolor{black}{For}} $k=0,1,2,\ldots,K-2$ \textbf{\textcolor{black}{do}} 

$\qquad$Obtain $x_{k+1}$ using (\ref{inexactprimalupdate}),

$\qquad$Update $\varLambda_{k+1}$ using (\ref{dualupdateexplicit}).

\textbf{\textcolor{black}{end for}} 

Set $\bar{x}_{K}=\frac{1}{K}\sum_{k=0}^{K-1}x_{k}$ and $\bar{\varLambda}_{K}=\frac{1}{K}\sum_{k=0}^{K-1}\varLambda_{k}$.

\textbf{\textcolor{black}{Return:}} $\bar{x}_{K}$.
\end{algorithm}

We now proceed to analyze the output of Algorithm \ref{alg:The-Inexact-Primal-Dual}.
For any probability measure $\phi\in\mathcal{P}\left(\Xi\right)$
and any $x\in\mathcal{X}$, we define $\mathbb{E}_{\phi}[g\left(x,\,\xi\right)]\triangleq\left\langle \phi,G(x)\right\rangle $.
First, we observe that the special structure of Problem (\ref{dualupdate})
yields an analytical solution by the calculus of variations. 
\begin{thm}
\label{dualupdatesolution} The optimal solution of Problem (\ref{dualupdate})
is given by

\begin{equation}
\begin{array}{rl}
\varLambda_{k+1}(\xi)=\, & \min\left\{ \bar{\rho}(\theta)/\left(\left(\rho_{0}/\mathrm{vol}\left(\Xi\right)\right)^{\frac{\gamma\,\kappa}{1+\gamma\,\kappa}}\intop_{\Xi}\exp\left(\frac{\gamma\,g(x_{k},\,\xi)}{1+\gamma\,\kappa}\right)(\varLambda_{k}(\xi))^{\frac{1}{1+\gamma\,\kappa}}d\xi\right),1\right\} \\
 & \times\left(\rho_{0}/\text{vol}\left(\Xi\right)\right)^{\frac{\gamma\,\kappa}{1+\gamma\,\kappa}}\exp\left(\frac{\gamma\,g(x_{k},\,\xi)}{1+\gamma\,\kappa}\right)(\varLambda_{k}(\xi))^{\frac{1}{1+\gamma\,\kappa}},
\end{array}\label{dualupdateexplicit}
\end{equation}
for all $\xi\in\Xi$.
\end{thm}

\begin{proof}
There is a one-to-one correspondence between $\varLambda\in\mathcal{V}$
and pairs $(\rho(\varLambda),\,\phi(\varLambda))\in[0,\,\bar{\rho}(\theta)]\times\mathcal{P}(\Xi)$
by Lemma \ref{lem:prox_preliminary}(i). From Theorem \ref{propertiesofB}(i),
we directly obtain

\[
\begin{array}{rcl}
 &  & \gamma\left(\left\langle \varLambda,\,-G\left(x_{k}\right)\right\rangle +\kappa\,B\left(\varLambda,\,\varLambda_{u}\right)\right)+B\left(\varLambda,\,\varLambda_{k}\right)\\
 & = & \rho(\varLambda)\left[D(\phi(\varLambda),\,\phi(\varLambda_{k}))+\gamma\kappa\,D(\phi(\varLambda),\phi_{u})-\gamma\mathbb{E}_{\phi(\varLambda)}[g\left(x_{k},\,\xi\right)]\right]+\gamma\kappa\,H(\rho(\varLambda),\,\rho_{0})+H(\rho(\varLambda),\,\rho(\varLambda_{k})),
\end{array}
\]
which implies that the optimal solution of (\ref{first-y-update})
over $\phi(\varLambda)\in\mathcal{\mathcal{P}}(\Xi)$ does not depend
on the outer variable $\rho(\varLambda)$. Thus, the optimization
Problem (\ref{dualupdate}) over $\varLambda\in\mathcal{V}$ is equivalent
to solving a two-stage optimization problem: in the first stage, we
optimize over $\phi\in\mathcal{\mathcal{P}}(\Xi)$ to obtain $\phi(\varLambda_{k+1})$,

\begin{equation}
\phi(\varLambda_{k+1})\in\arg\min_{\phi\in\mathcal{\mathcal{P}}(\Xi)}\left[D(\phi,\,\phi(\varLambda_{k}))+\gamma\kappa\,D(\phi,\,\phi_{u})-\gamma\mathbb{E}_{\phi}[g\left(x_{k},\,\xi\right)]\right],\label{first-y-update}
\end{equation}
and in the second stage, we optimize over $\rho\in[0,\,\bar{\rho}(\theta)]$
to obtain $\rho(\varLambda_{k+1})$,

\begin{equation}
\begin{array}{rcl}
\rho(\varLambda_{k+1})\in & \arg\min_{\rho\in[0,\,\bar{\rho}(\theta)]} & \rho\left[D(\phi(\varLambda_{k+1}),\,\phi(\varLambda_{k}))+\gamma\kappa\,D(\phi(\varLambda_{k+1}),\,\phi_{u})-\gamma\mathbb{E}_{\phi(\varLambda_{k+1})}[g\left(x_{k},\,\xi\right)]\right]\\
 &  & +\gamma\kappa\,H(\rho,\,\rho_{0})+H(\rho,\,\rho(\varLambda_{k})).
\end{array}\label{second-rho-update}
\end{equation}
The measure $\rho(\varLambda_{k+1})\phi(\varLambda_{k+1})\in\mathcal{V}$
is then a solution of Problem (\ref{dualupdate}). Problem (\ref{first-y-update})
is a convex optimization problem due to the convexity of $\phi\rightarrow D(\phi,\,\varphi)$
for all fixed $\varphi$ (see \cite{mackay2003information}). Furthermore,
Problem (\ref{first-y-update}) is a constrained calculus of variations
problem:

\[
\begin{array}{rcl}
\min_{\phi\in\mathcal{M}_{+}\left(\Xi\right)} &  & \intop_{\Xi}\log\left(\frac{\phi(\xi)}{\phi(\varLambda_{k})(\xi)}\right)\phi(\xi)+\gamma\kappa\log\left(\phi(\xi)\text{vol}\left(\Xi\right)\right)\phi(\xi)-\gamma g\left(x_{k},\,\xi\right)\phi(\xi)d\xi\\
s.t. &  & \intop_{\Xi}\phi(\xi)d\xi=1.
\end{array}
\]
Using Euler's equation (see \cite[Section 7.5]{luenberger1997optimization}),
we obtain

\begin{equation}
(1+\gamma\kappa)\log\left(\phi(\xi)\right)=\gamma g\left(x_{k},\,\xi\right)+\log\left(\phi(\varLambda_{k})(\xi)\right)+C,\,\forall\xi\in\Xi,\label{Euler}
\end{equation}
where $C=-\gamma\kappa\log(\text{vol}\left(\Xi\right))-\gamma\kappa-1-\upsilon$
and $\upsilon\in\mathbb{R}$ is the Lagrange multiplier of the equality
constraint $\intop_{\Xi}\phi(\xi)d\xi=1$. From (\ref{Euler}) and
the constraint $\intop_{\Xi}\phi(\xi)d\xi=1$, we obtain

\[
\phi(\varLambda_{k+1})(\xi)=\frac{(\phi(\varLambda_{k})(\xi))^{(1+\gamma\kappa)^{-1}}\exp\left(\gamma g(x_{k},\,\xi)/(1+\gamma\kappa)\right)}{\intop_{\Xi}(\phi(\varLambda_{k})(\xi))^{(1+\gamma\kappa)^{-1}}\exp\left(\gamma g(x_{k},\,\xi)/(1+\gamma\kappa)\right)d\xi},\:\forall\xi\in\Xi.
\]

Problem (\ref{second-rho-update}) is a convex optimization problem
due to  the convexity of $\rho\rightarrow H(\rho,\,\rho')$ for all
fixed $\rho'>0$ (see Lemma \ref{lem:prox_preliminary}). By direct
calculation and simplification, we obtain 

\[
\begin{array}{rcl}
 &  & D(\phi(\varLambda_{k+1}),\,\phi(\varLambda_{k}))+\gamma\kappa\,D(\phi(\varLambda_{k+1}),\phi_{u})-\gamma\mathbb{E}_{\phi(\varLambda_{k+1})}[g\left(x_{k},\,\xi\right)]\\
 & = & -(1+\gamma\kappa)\log\left(\intop_{\Xi}(\phi(\varLambda_{k})(\xi))^{\frac{1}{1+\gamma\kappa}}\exp\left(\frac{\gamma g(x_{k},\,\xi)}{1+\gamma\kappa}\right)d\xi\right)+\gamma\kappa\log(\text{vol}\left(\Xi\right)).
\end{array}
\]
Define the following function of the scalar variable $\rho\geq0$, 

\[
\begin{array}{rcl}
F(\rho) & \triangleq & -\left[(1+\gamma\kappa)\log\left(\intop_{\Xi}(\phi(\varLambda_{k})(\xi))^{\frac{1}{1+\gamma\kappa}}\exp\left(\frac{\gamma g(x_{k},\,\xi)}{1+\gamma\kappa}\right)d\xi\right)-\gamma\kappa\log(\text{vol}\left(\Xi\right))\right]\rho\\
 &  & +\gamma\kappa\,H(\rho,\,\rho_{0})+H(\rho,\,\rho(\varLambda_{k})).
\end{array}
\]
Then, we may write Problem (\ref{second-rho-update}) as $\rho(\varLambda_{k+1})\in\arg\min_{\rho\in[0,\,\bar{\rho}(\theta)]}F(\rho)$.
From the calculation $\frac{\partial H(\rho,\,\rho')}{\partial\rho}=\log(\frac{\rho}{\rho'})$,
we observe that $F'(\rho)$ is increasing in $\rho$, $\lim_{\rho\downarrow0}F'(\rho)=-\infty$,
and $\lim_{\rho\uparrow\infty}F'(\rho)=\infty$. Therefore, we obtain

\[
\rho(\varLambda_{k+1})=\min\left\{ (\rho_{0}\text{vol}\left(\Xi\right)^{-1})^{\frac{\gamma\kappa}{1+\gamma\kappa}}(\rho(\varLambda_{k}))^{\frac{1}{1+\gamma\kappa}}\intop_{\Xi}(\phi(\varLambda_{k})(\xi))^{\frac{1}{1+\gamma\kappa}}\exp\left(\frac{\gamma g(x_{k},\,\xi)}{1+\gamma\kappa}\right)d\xi,\,\bar{\rho}(\theta)\right\} .
\]
We complete the proof by selecting $\varLambda_{k+1}=\rho(\varLambda_{k+1})\phi(\varLambda_{k+1})$. 
\end{proof}
The next theorem gives general error bounds for our inexact primal-dual
algorithm. Specifically, it bounds the optimality gap and constraint
violation of the averaged solution $\overline{x}_{K}$. The term 
\[
\mu(\epsilon,\theta)\triangleq\sqrt{2(C(\epsilon,\,\theta)+D_{\mathcal{X}})\left(\bar{\rho}(\theta)G_{\max}^{2}+2(L_{f}+\bar{\rho}(\theta)L_{g,\,\mathcal{X}})^{2}\right)}
\]
will feature prominently in these bounds. Naturally, $\mu(\epsilon,\theta)$
is increasing in $\theta$, $D_{\mathcal{X}}$, $D_{\Xi}$, $G_{\max},$
$L_{f},$ $L_{g,\,\mathcal{X}}$, $L_{g,\,\Xi}$, and $d$, and is
decreasing in $\epsilon$.
\begin{thm}
\label{PSMDthm} Choose $\epsilon>0$, $K\geq1$ and let $\bar{x}_{K}$
and $\overline{\varLambda}_{K}$ be produced by Algorithm \ref{alg:The-Inexact-Primal-Dual}.
Then,
\end{thm}

\[
f(\overline{x}_{K})-f(x^{\ast})\leq\frac{\mu(\epsilon,\theta)}{\sqrt{K}}+\frac{\sqrt{D_{\mathcal{X}}}}{K}\sum_{k=0}^{K-1}\left\Vert \varepsilon_{k}\right\Vert +\frac{3\,\gamma}{2\,K}\sum_{k=0}^{K-1}\left\Vert \varepsilon_{k}\right\Vert ^{2}+\epsilon,
\]

\[
g(\overline{x}_{K},\,\xi)\leq\frac{\mu(\epsilon,\theta)}{\bar{\rho}(\theta)\sqrt{K}}+\frac{\sqrt{D_{\mathcal{X}}}}{\bar{\rho}(\theta)\,K}\sum_{k=0}^{K-1}\left\Vert \varepsilon_{k}\right\Vert +\frac{3\,\gamma}{2\,\bar{\rho}(\theta)\,K}\sum_{k=0}^{K-1}\left\Vert \varepsilon_{k}\right\Vert ^{2}+\frac{\epsilon}{\bar{\rho}(\theta)},\quad\forall\xi\in\Xi.
\]

\begin{proof}
By fixing $x=x^{\ast}$ and $\varLambda=0$ in (\ref{violation})
in Lemma \ref{twosides} (which bounds the difference $L_{\bar{\kappa}}(\overline{x}_{K},\,\varLambda)-L_{\bar{\kappa}}(x,\,\overline{\varLambda}_{K})$
for all $x\in\mathcal{X}$ and $\varLambda\in\mathcal{V}$ with a
continuous density), and noting that $B(0,\,\varLambda_{u})=\rho_{0}\leq\max(\rho_{0},\,C'(\epsilon,\,\theta))=C(\epsilon,\,\theta)$,
$\bar{\kappa}(\epsilon)\,B\left(\overline{\varLambda}_{K},\,\varLambda_{u}\right)\geq0$,
and $G(x^{\ast})\leq0$, we have

\begin{equation}
\begin{array}{rcl}
f(\overline{x}_{K})-f(x^{\ast}) & \leq & \frac{1}{\gamma K}(C(\epsilon,\,\theta)+D_{\mathcal{X}})+\frac{\gamma}{2}(\bar{\rho}(\theta)G_{\max}^{2}+2(L_{f}+\bar{\rho}(\theta)L_{g,\,\mathcal{X}})^{2})\\
 &  & +\frac{3}{2}\gamma\,\frac{1}{K}\sum_{k=0}^{K-1}\left\Vert \varepsilon_{k}\right\Vert ^{2}+\sqrt{D_{\mathcal{X}}}\,\frac{1}{K}\sum_{k=0}^{K-1}\left\Vert \varepsilon_{k}\right\Vert +\bar{\kappa}(\epsilon)\rho_{0}\\
 & \leq & \frac{\mu(\epsilon,\theta)}{\sqrt{K}}+\epsilon+\frac{3}{2}\gamma\,\frac{1}{K}\sum_{k=0}^{K-1}\left\Vert \varepsilon_{k}\right\Vert ^{2}+\sqrt{D_{\mathcal{X}}}\,\frac{1}{K}\sum_{k=0}^{K-1}\left\Vert \varepsilon_{k}\right\Vert ,
\end{array}\label{optimalitygap}
\end{equation}
where the second inequality follows from the choice of $\gamma$ in
Algorithm \ref{alg:The-Inexact-Primal-Dual} and $\bar{\kappa}(\epsilon)\leq\epsilon/\rho_{0}$
by definition of $\bar{\kappa}(\epsilon)$.

To bound the constraint violation, we fix $x=x_{\bar{\kappa}(\epsilon)}^{\ast}$
and $\varLambda=\varLambda_{\bar{\kappa}(\epsilon)}(\overline{x}_{K})$
 in (\ref{violation}). By Lemma \ref{LemmaQB} (which upper bounds
$B(\varLambda_{\kappa}(x),\,\varLambda_{u})$ for all $x\in\mathcal{X}$),
we have 
\[
B(\varLambda_{\bar{\kappa}(\epsilon)}(\overline{x}_{K}),\,\varLambda_{u})\leq C'(\epsilon,\,\theta)\leq\max(\rho_{0},\,C'(\epsilon,\,\theta))=C(\epsilon,\,\theta).
\]
Therefore, we have,

\begin{equation}
\begin{array}{rcl}
 &  & L_{\bar{\kappa}(\epsilon)}(\overline{x}_{K},\,\varLambda_{\bar{\kappa}(\epsilon)}(\overline{x}_{K}))-L_{\bar{\kappa}(\epsilon)}(x_{\bar{\kappa}(\epsilon)}^{\ast},\,\overline{\varLambda}_{K})\\
 & \leq & \frac{1}{\gamma K}C(\epsilon,\,\theta)+\frac{1}{\gamma K}D_{\mathcal{X}}+\frac{\gamma}{2}(\bar{\rho}(\theta)G_{\max}^{2}+2(L_{f}+\bar{\rho}(\theta)L_{g,\,\mathcal{X}})^{2})+\frac{3}{2}\gamma\,\frac{1}{K}\sum_{k=0}^{K-1}\left\Vert \varepsilon_{k}\right\Vert ^{2}+\sqrt{D_{\mathcal{X}}}\,\frac{1}{K}\sum_{k=0}^{K-1}\left\Vert \varepsilon_{k}\right\Vert \\
 & = & \frac{\mu(\epsilon,\theta)}{\sqrt{K}}+\frac{3}{2}\gamma\,\frac{1}{K}\sum_{k=0}^{K-1}\left\Vert \varepsilon_{k}\right\Vert ^{2}+\sqrt{D_{\mathcal{X}}}\,\frac{1}{K}\sum_{k=0}^{K-1}\left\Vert \varepsilon_{k}\right\Vert .
\end{array}\label{violationofconstr1}
\end{equation}
On the other hand, 

\begin{equation}
\begin{array}{rcl}
 &  & L_{\bar{\kappa}(\epsilon)}(\overline{x}_{K},\,\varLambda_{\bar{\kappa}(\epsilon)}(\overline{x}_{K}))-L_{\bar{\kappa}(\epsilon)}(x_{\bar{\kappa}(\epsilon)}^{\ast},\,\overline{\varLambda}_{K})\\
 & \geq & f(\overline{x}_{K})+\max_{\varLambda\in\mathcal{V}}\left\langle \varLambda,\,G(\overline{x}_{K})\right\rangle -\epsilon-f(x_{\bar{\kappa}(\epsilon)}^{\ast})-\left\langle \overline{\varLambda}_{K},\,G(x_{\bar{\kappa}(\epsilon)}^{\ast})\right\rangle +\bar{\kappa}(\epsilon)\,B\left(\overline{\varLambda}_{K},\,\varLambda_{u}\right)\\
 & \geq & \max_{\varLambda\in\mathcal{V}}\left\langle \varLambda,\,G(\overline{x}_{K})\right\rangle -\epsilon+\left\langle \varLambda_{\bar{\kappa}(\epsilon)}^{\ast},\,G(x_{\bar{\kappa}(\epsilon)}^{\ast})\right\rangle -\left\langle \varLambda_{\bar{\kappa}(\epsilon)}^{\ast},\,G(\overline{x}_{K})\right\rangle -\left\langle \overline{\varLambda}_{K},\,G(x_{\bar{\kappa}(\epsilon)}^{\ast})\right\rangle +\bar{\kappa}(\epsilon)\,B\left(\overline{\varLambda}_{K},\,\varLambda_{u}\right)\\
 & \geq & \max_{\varLambda\in\mathcal{V}}\left\langle \varLambda,\,G(\overline{x}_{K})\right\rangle -\epsilon+\bar{\kappa}(\epsilon)\,B\left(\varLambda_{\bar{\kappa}(\epsilon)}^{\ast},\,\varLambda_{u}\right)\\
 & \geq & \bar{\rho}(\theta)g(\overline{x}_{K},\,\bar{\xi})-\epsilon,
\end{array}\label{voilationofconstr2}
\end{equation}
where the first inequality follows from Theorem \ref{maxsupdiff_less than_epsilon}
(which shows that the gap between the values of the inner maximization
over $\mathcal{V}$ of the regularized Lagrangian and the original
Lagrangian can be made arbitrarily small through our control of $\epsilon$,
i.e., $\max_{\varLambda\in\mathcal{V}}\left\{ \left\langle \varLambda,\,G(x)\right\rangle -\bar{\kappa}(\epsilon)\,B\left(\varLambda,\,\varLambda_{u}\right)\right\} \geq\max_{\varLambda\in\mathcal{V}}\left\{ \left\langle \varLambda,\,G(x)\right\rangle \right\} -\epsilon$
for any $x\in\mathcal{X}$), the second inequality follows from $f(x_{\bar{\kappa}(\epsilon)}^{\ast})+\left\langle \varLambda_{\bar{\kappa}(\epsilon)}^{\ast},\,G(x_{\bar{\kappa}(\epsilon)}^{\ast})\right\rangle \leq f(\overline{x}_{K})+\left\langle \varLambda_{\bar{\kappa}(\epsilon)}^{\ast},\,G(\overline{x}_{K})\right\rangle $
due to the definition of a saddle-point, the third is from the inequality
\[
\left\langle \varLambda_{\bar{\kappa}(\epsilon)}^{\ast},\,G(x_{\bar{\kappa}(\epsilon)}^{\ast})\right\rangle -\bar{\kappa}(\epsilon)\,B\left(\varLambda_{\bar{\kappa}(\epsilon)}^{\ast},\,\varLambda_{u}\right)\geq\left\langle \overline{\varLambda}_{K},\,G(x_{\bar{\kappa}(\epsilon)}^{\ast})\right\rangle -\bar{\kappa}(\epsilon)\,B\left(\overline{\varLambda}_{K},\,\varLambda_{u}\right),
\]
also due to the definition of a saddle-point, and the last inequality
is due to the specific selection $\varLambda=\bar{\rho}(\theta)\delta_{\bar{\xi}}\in\mathcal{V}$,
where $\bar{\xi}\in\Xi$ is arbitrary and $\bar{\kappa}(\epsilon)\,B\left(\varLambda_{\bar{\kappa}(\epsilon)}^{\ast},\,\varLambda_{u}\right)\geq0$.

Finally, from (\ref{violationofconstr1}) and (\ref{voilationofconstr2}),
we get

\begin{equation}
g(\overline{x}_{K},\,\bar{\xi})\leq\frac{\mu(\epsilon,\theta)}{\bar{\rho}(\theta)\sqrt{K}}+\frac{\epsilon}{\bar{\rho}(\theta)}+\frac{3\gamma}{2\bar{\rho}(\theta)}\,\frac{1}{K}\sum_{k=0}^{K-1}\left\Vert \varepsilon_{k}\right\Vert ^{2}+\frac{\sqrt{D_{\mathcal{X}}}}{\bar{\rho}(\theta)}\,\frac{1}{K}\sum_{k=0}^{K-1}\left\Vert \varepsilon_{k}\right\Vert ,\,\forall\bar{\xi}\in\Xi,\label{violationconstraint-barx_T}
\end{equation}
which completes the proof.
\end{proof}
\begin{rem}
(i) There is a tradeoff in these bounds in $\epsilon$ and $\theta$,
since $\mu(\epsilon,\theta)$ is decreasing in $\epsilon$ and increasing
in $\theta$.

(ii) Suppose $\|\varepsilon_{k}\|\leq\varepsilon$ for all $k=0,\,1,\ldots,\,K-1$,
then our algorithm achieves an $\mathcal{O}(1/\sqrt{K})$ rate of
convergence with overall error that depends on $\varepsilon$.

(iii) Suppose $\|\varepsilon_{k}\|\leq1/\left(k+1\right)$ for all
$k=0,\,1,\ldots,\,K-1$, then we have $\frac{1}{K}\sum_{k=0}^{K-1}\left\Vert \varepsilon_{k}\right\Vert ^{2}\leq\pi^{2}/\left(6\,K\right)$
since $\sum_{k=1}^{\infty}1/k^{2}=\pi^{2}/6$ and $\frac{1}{K}\sum_{k=0}^{K-1}\left\Vert \varepsilon_{k}\right\Vert \leq\left(\ln K+1\right)/K$,
due to the fact that $\sum_{k=1}^{K}1/k\leq1+\intop_{1}^{K}x^{-1}dx=\ln K+1$.
\end{rem}

\section{\label{sec:MC_integration} Approximation with Monte Carlo Integration}

In this section, we propose a specific scheme for approximating $\left\langle \varLambda_{k},\,\nabla G\left(x_{k}\right)\right\rangle $
based on uniformly sampling from $\Xi$. We begin by noting that
\[
\left\langle \varLambda_{k},\,\nabla G\left(x_{k}\right)\right\rangle =\int_{\Xi}\varLambda_{k}(\xi)\nabla_{x}g\left(x_{k},\,\xi\right)d\xi,
\]
which suggests approximation by Monte Carlo integration. 

Initially, we generate $N\geq1$ i.i.d. samples $\xi_{i}$ from $\Xi$,
$i=1,\,2,\ldots,\,N$, from the uniform distribution. Let $\hat{\varLambda}_{0}(\xi_{i})=\rho_{0}/\text{vol}\left(\Xi\right)$,
where $i=1,\,2,\ldots,\,N$. For $0\leq k\leq K-2$, $1\leq i\leq N$,
we then approximate $\varLambda_{k+1}(\xi_{i})$ by 

\begin{equation}
\begin{array}{rl}
\hat{\varLambda}_{k+1}(\xi_{i})=\, & \min\left\{ \bar{\rho}(\theta)/\left(\left(\rho_{0}/\text{vol}\left(\Xi\right)\right)^{\frac{\gamma\,\kappa}{1+\gamma\,\kappa}}\frac{\text{vol}\left(\Xi\right)}{N}\sum_{i=1}^{N}\exp\left(\frac{\gamma\,g(x_{k},\,\xi_{i})}{1+\gamma\,\kappa}\right)\hat{\varLambda}_{k}(\xi_{i})^{\frac{1}{1+\gamma\,\kappa}}\right),1\right\} \\
 & \times\left(\rho_{0}/\text{vol}\left(\Xi\right)\right)^{\frac{\gamma\,\kappa}{1+\gamma\,\kappa}}\exp\left(\frac{\gamma\,g(x_{k},\,\xi_{i})}{1+\gamma\,\kappa}\right)\hat{\varLambda}_{k}(\xi_{i})^{\frac{1}{1+\gamma\,\kappa}}.
\end{array}\label{MCintegration_for_dualupdate}
\end{equation}
Clearly, the $\hat{\varLambda}_{k}$ are all approximations of $\varLambda_{k}$,
and all $\hat{\varLambda}_{k}$ share the same support $\left\{ \xi_{1},\xi_{2},\ldots,\xi_{N}\right\} $.
Then, we approximate $\left\langle \varLambda_{k},\,\nabla G\left(x_{k}\right)\right\rangle $
with

\[
\mathcal{G}_{k}(x_{k},\varLambda_{k})=\frac{\text{vol}\left(\Xi\right)}{N}\sum_{n=1}^{N}\hat{\varLambda}_{k}(\xi_{i})\nabla_{x}g\left(x_{k},\,\xi_{i}\right).
\]
The modified primal update is

\begin{equation}
x_{k+1}=\arg\min_{x\in\mathcal{X}}\left[\gamma\langle x-x_{k},\,\nabla f(x_{k})+\frac{\text{vol}\left(\Xi\right)}{N}\sum_{n=1}^{N}\hat{\varLambda}_{k}(\xi_{i})\nabla_{x}g\left(x_{k},\,\xi_{i}\right)\rangle+\frac{1}{2}\|x-x_{k}\|^{2}\right].\label{stochasticprimalupdate}
\end{equation}

\begin{algorithm}
\caption{\label{alg:The-Randomized-Primal-Dual}The Primal-Dual Algorithm with
Monte Carlo Integration}

\textbf{\textcolor{black}{Input: }}Number of iterations $K\geq1$,
number of samples $N\geq1$, initial points $x_{0}\in\mathcal{X},$
and constant step-size 

\[
\gamma=\sqrt{\frac{2(C(\epsilon,\,\theta)+D_{\mathcal{X}})}{\left(\bar{\rho}(\theta)G_{\max}^{2}+2(L_{f}+\bar{\rho}(\theta)L_{g,\,\mathcal{X}})^{2}\right)K}}.
\]

Generate $N\geq1$ i.i.d. samples $\xi_{i}$ from $\Xi$ from the
uniform distribution, and $\hat{\varLambda}_{0}(\xi_{i})=\rho_{0}/\text{vol}\left(\Xi\right)$,
where $i=1,\,2,\ldots,\,N$. 

\textbf{\textcolor{black}{For}} $k=0,1,2,\ldots,K-2$ \textbf{\textcolor{black}{do}} 

$\qquad$Obtain $x_{k+1}$ using (\ref{stochasticprimalupdate}), 

$\qquad$Obtain $\hat{\varLambda}_{k+1}$ using (\ref{MCintegration_for_dualupdate}).

\textbf{\textcolor{black}{end for}} 

Set $\bar{x}_{K}=\frac{1}{K}\sum_{k=0}^{K-1}x_{k}$. 

\textbf{\textcolor{black}{Return:}} $\overline{x}_{K}$.
\end{algorithm}

Denote $m_{u}\triangleq\rho_{0}/\text{vol}\left(\Xi\right)$, $\lambda_{k}\triangleq\varLambda_{k}/m_{u}$,
and $\hat{\lambda}_{k}\triangleq\hat{\varLambda}_{k}/m_{u}$, $\rho\triangleq\bar{\rho}(\theta)/m_{u}$,
$s_{k}(\xi)\triangleq\exp\left(\frac{\gamma\,g(x_{k},\,\xi)}{1+\gamma\,\kappa}\right)$,
$l\triangleq\frac{1}{1+\gamma\,\kappa}$. From (\ref{dualupdateexplicit})
and (\ref{MCintegration_for_dualupdate}), we have 

\[
\lambda_{k+1}(\xi)=\min\left\{ \rho/\left(\intop_{\Xi}s_{k}(\xi)\lambda_{k}(\xi)^{l}d\xi\right),1\right\} s_{k}(\xi)\lambda_{k}(\xi)^{l},
\]

\[
\hat{\lambda}_{k+1}(\xi_{i})=\min\left\{ \rho/\left(\frac{\text{vol}\left(\Xi\right)}{N}\sum_{i=1}^{N}s_{k}(\xi_{i})\hat{\lambda}_{k}(\xi_{i})^{l}\right),1\right\} s_{k}(\xi_{i})\hat{\lambda}_{k}(\xi_{i})^{l}.
\]

Clearly, $\lambda_{0}(\xi)=1$ for any $\xi\in\Xi$, and $\hat{\lambda}_{0}(\xi_{i})=1$
for all $i=1,\,2,\ldots,\,N$. For further analysis, we need an upper
bound on $\lambda_{k}(\xi)$ for $k=0,1,\ldots,K-1$ and $\xi\in\Xi$.
\begin{assumption}
\label{upperbound_M}There exists $M>1$, such that $\lambda_{k}(\xi)\leq M$
for $k=0,1,\ldots,K-1$ and $\xi\in\Xi$.
\end{assumption}

\begin{rem}
Since $g(x,\,\xi)\leq G_{\max}$ for all $x\in\mathcal{X}$ and $\xi\in\Xi$,
we have $s_{k}(\xi)\leq\exp\left(\frac{\gamma\,G_{\max}}{1+\gamma\,\kappa}\right)$
for all $k=0,1,\ldots,K-1$ and $\xi\in\Xi$. We can further deduce
that for any $k\geq1$ and $\xi\in\Xi$, 

\[
\lambda_{k}(\xi)\leq\exp\left(\frac{\gamma\,G_{\max}}{1+\gamma\,\kappa}\sum_{j=0}^{k-1}l^{j}\right)\leq\exp\left(\frac{\gamma\,G_{\max}}{(1+\gamma\,\kappa)(1-l)}\right)=\exp\left(\frac{G_{\max}}{\kappa}\right).
\]
Therefore, we can take $M=\exp\left(\frac{G_{\max}}{\kappa}\right)$
to satisfy Assumption \ref{upperbound_M}. 
\end{rem}

Define $R(r)=(1+\beta)(\frac{1+r}{1-r})^{l}-1$ where $\beta>0$ and
$r\in(0,1)$, and denote $R^{k}(r)$ as the $k^{th}$ iterate of $R$
where $k$ is a non-negative integer. Lemma \ref{increasingproperty}
claims that if $l\in[\frac{1}{2},1)$, for any $\varepsilon>0$ and
positive integer $K$, there exists $\eta(\varepsilon,K)\in(0,\varepsilon)$,
such that as long as $\beta\leq\eta(\varepsilon,K)$, we have $R^{k}(\beta)\leq\varepsilon$
for all $k=1,2,\ldots,K-1$. The next theorem gives error bounds for
our primal-dual algorithm based on Monte Carlo integration. 
\begin{thm}
\label{errorbounds_MonteCarlo}For any $\epsilon>0$, $\varepsilon>0$,
$\delta\in(0,1)$, $K\geq\frac{2(C(\epsilon,\,\theta)+D_{\mathcal{X}})\bar{\kappa}(\epsilon)^{2}}{\bar{\rho}(\theta)G_{\max}^{2}+2(L_{f}+\bar{\rho}(\theta)L_{g,\,\mathcal{X}})^{2}}$,
and sample size

\[
N\geq\max\left\{ \frac{2L_{g,\mathcal{X}}^{2}M^{2}\rho_{0}^{2}}{\varepsilon^{2}}\ln(\frac{4K}{\delta}),\frac{1}{2\eta(\frac{\varepsilon}{2\rho_{0}L_{g,\mathcal{X}}M},K)^{2}}\ln(\frac{4K}{\delta})\right\} ,
\]
we have

\[
f(\overline{x}_{K})-f(x^{\ast})\leq\frac{\mu(\epsilon,\theta)}{\sqrt{K}}+\sqrt{D_{\mathcal{X}}}\varepsilon+\frac{3\,\gamma}{2\,}\varepsilon^{2}+\epsilon,
\]

\[
g(\overline{x}_{K},\,\xi)\leq\frac{\mu(\epsilon,\theta)}{\bar{\rho}(\theta)\sqrt{K}}+\frac{\sqrt{D_{\mathcal{X}}}}{\bar{\rho}(\theta)\,}\varepsilon+\frac{3\,\gamma}{2\,\bar{\rho}(\theta)\,}\varepsilon^{2}+\frac{\epsilon}{\bar{\rho}(\theta)},\quad\forall\xi\in\Xi,
\]
with probability at least $1-\delta$.
\end{thm}

\begin{proof}
The approximation error at iteration $k\geq0$ is 

\[
\begin{array}{rcl}
\varepsilon_{k} & = & m_{u}\frac{\text{vol}\left(\Xi\right)}{N}\sum_{i=1}^{N}\hat{\lambda}_{k}(\xi_{i})\nabla_{x}g\left(x_{k},\,\xi_{i}\right)-m_{u}\int_{\Xi}\lambda_{k}(\xi)\nabla_{x}g\left(x_{k},\,\xi\right)d\xi\\
 & = & \frac{\rho_{0}}{N}\sum_{i=1}^{N}\lambda_{k}(\xi_{i})\left(\frac{\hat{\lambda}_{k}(\xi_{i})}{\lambda_{k}(\xi_{i})}-1\right)\nabla_{x}g\left(x_{k},\,\xi_{i}\right)\\
 &  & +\rho_{0}\left(\frac{1}{N}\sum_{i=1}^{N}\lambda_{k}(\xi_{i})\nabla_{x}g\left(x_{k},\,\xi_{i}\right)-\int_{\Xi}\lambda_{k}(\xi)\nabla_{x}g\left(x_{k},\,\xi\right)/\text{vol}\left(\Xi\right)d\xi\right),
\end{array}
\]
thus we have

\begin{equation}
\begin{array}{rcl}
\left\Vert \varepsilon_{k}\right\Vert  & \leq & \rho_{0}L_{g,\mathcal{X}}M_{2}\max_{i=1,\,2,\ldots,\,N}\left|\frac{\hat{\lambda}_{k}(\xi_{i})}{\lambda_{k}(\xi_{i})}-1\right|\\
 &  & +\rho_{0}\left|\frac{1}{N}\sum_{i=1}^{N}\lambda_{k}(\xi_{i})\nabla_{x}g\left(x_{k},\,\xi_{i}\right)-\int_{\Xi}\lambda_{k}(\xi)\nabla_{x}g\left(x_{k},\,\xi\right)/\text{vol}\left(\Xi\right)d\xi\right|.
\end{array}\label{errorinMCinteg}
\end{equation}

From Lemma \ref{sample_complexity_lemma}, we have

\[
\left|\frac{1}{N}\sum_{i=1}^{N}\lambda_{k}(\xi_{i})\nabla_{x}g\left(x_{k},\,\xi_{i}\right)-\int_{\Xi}\lambda_{k}(\xi)\nabla_{x}g\left(x_{k},\,\xi\right)/\text{vol}\left(\Xi\right)d\xi\right|\leq\frac{\varepsilon}{2\rho_{0}},\quad k=0,1,\ldots,K-1,
\]

\[
\max_{i=1,\,2,\ldots,\,N}\left|\frac{\hat{\lambda}_{k}(\xi_{i})}{\lambda_{k}(\xi_{i})}-1\right|\leq\frac{\varepsilon}{2\rho_{0}L_{g,\mathcal{X}}M},\quad k=0,1,\ldots,K-1,
\]
with probability at least $1-\delta$. Together with (\ref{errorinMCinteg}),
we obtain

\[
\left\Vert \varepsilon_{k}\right\Vert \leq\varepsilon\quad k=0,1,\ldots,K-1
\]
with probability at least $1-\delta$, and finally arrive at the conclusion
by a direct application of Theorem \ref{PSMDthm}. 
\end{proof}
\begin{rem}
(i) From the inequality $K\geq\frac{2(C(\epsilon,\,\theta)+D_{\mathcal{X}})\bar{\kappa}(\epsilon)^{2}}{\bar{\rho}(\theta)G_{\max}^{2}+2(L_{f}+\bar{\rho}(\theta)L_{g,\,\mathcal{X}})^{2}}$,
we have the stepsize $\gamma\leq1/\bar{\kappa}(\epsilon)$ and $l\in[\frac{1}{2},1)$,
and so we can invoke Lemma \ref{sample_complexity_lemma}. 

(ii) From the proof of Lemma \ref{sample_complexity_lemma}, we see
that $N\geq\frac{2L_{g,\mathcal{X}}^{2}M^{2}\rho_{0}^{2}}{\varepsilon^{2}}\ln(\frac{4K}{\delta})$
guarantees that for each $k=0,1,\ldots,K-1$, $\left|\frac{1}{N}\sum_{i=1}^{N}\lambda_{k}(\xi_{i})\nabla_{x}g\left(x_{k},\,\xi_{i}\right)-\int_{\Xi}\lambda_{k}(\xi)\nabla_{x}g\left(x_{k},\,\xi\right)/\text{vol}\left(\Xi\right)d\xi\right|\leq\frac{\varepsilon}{2\rho_{0}}$
with probability at least $1-\frac{\delta}{2K}$. Moreover, $N\geq\frac{1}{2\eta((\varepsilon/2\rho_{0}L_{g,\mathcal{X}}M),K)^{2}}\ln(\frac{4K}{\delta})$
guarantees that $\max_{i=1,\,2,\ldots,\,N}\left|\frac{\hat{\lambda}_{k}(\xi_{i})}{\lambda_{k}(\xi_{i})}-1\right|\leq\frac{\varepsilon}{2\rho_{0}L_{g,\mathcal{X}}M}$
for all $k=1,2,\ldots,K-1$ with probability at least $1-\delta/2$. 

(iii) The required number of samples $N$ is independent of the dimension
$d$ of the constraint index set $\Xi$. 
\end{rem}

\section{\label{sec:Proofs-of-Main Results}Proofs of Supporting Results}

We organize the proofs of all supporting results in this section. 
\begin{enumerate}
\item Subsection \ref{subsec:Problem SP and Dual Bound} shows the relation
between a saddle-point of Problem $\mathbb{SP}$ and the solutions
of Problem $\mathbb{P}$ and Problem $\mathbb{D}$, and shows that
we can restrict the dual feasible region to within a bounded set. 
\item In Subsection \ref{subsec:Properties-of B}, we show that $B(\cdot,\,\cdot)$
has all the properties required of a prox function (see \cite[Section 3.2]{duchi2016introductory}). 
\item In Subsection \ref{subsec:Results-used-in Fan}, we give the proof
for Theorem \ref{existenceofSP_R} which establishes existence of
a saddle-point of Problem $\mathbb{SP}_{\kappa}$, and we provide
all supporting results for the proof of Theorem \ref{maxsupdiff_less than_epsilon}
(which demonstrates the gap between the values of the inner maximization
over $\mathcal{V}$ of the regularized Lagrangian and the original
Lagrangian can be made arbitrarily small). 
\item Subsection \ref{subsec:Convex-concave analysis} provides a convex-concave
analysis for the primal-dual iterates for the regularized saddle-point
problem, and gives a bound on the difference $L_{\kappa}(\overline{x}_{K},\,\varLambda)-L_{\kappa}(x,\,\overline{\varLambda}_{K})$,
which paves the way for the proof of Theorem \ref{PSMDthm} for our
inexact primal-dual algorithm. 
\item In Subsection \ref{subsec:Sample-Complexity-ofMC_integration}, we
provide the sample complexity analysis of Monte Carlo integration,
which is the key ingredient in the proof of Theorem \ref{errorbounds_MonteCarlo}
for our primal-dual algorithm based on Monte Carlo integration. 
\end{enumerate}

\subsection{\label{subsec:Problem SP and Dual Bound}The Dual Problem}

In this subsection, we provide the proofs for Theorem \ref{SP} and
Theorem \ref{dualupperbound}. Theorem \ref{SP} connects the solution
of Problem $\mathbb{SP}$ with the primal and dual Problems $\mathbb{P}$
and $\mathbb{D}$.
\begin{proof}[Proof of Theorem \ref{SP}]
 (i) The first part is a direct consequence of \cite[Section 8.3, Corollary 1]{luenberger1997optimization}.

(ii) Let $\left(x^{*},\,\varLambda^{*}\right)\in\mathcal{X}\times\mathcal{M}_{+}\left(\Xi\right)$
be a saddle-point of Problem $\mathbb{SP}$, i.e., 

\[
L(x^{*},\,\varLambda)\leq L(x^{*},\,\varLambda^{*})\leq L(x,\,\varLambda^{*}),\:\forall(x,\,\varLambda)\in\mathcal{X}\times\mathcal{M}_{+}\left(\Xi\right).
\]

We prove that $x^{*}$ is primal optimal. We argue that $x^{*}$ is
feasible for Problem $\mathbb{P}$ as follows. Indeed, if $g\left(x^{*},\,\xi_{0}\right)>0$,
for some $\xi_{0}\in\Xi$, then we can select the measure $\varLambda=\varLambda^{*}+\delta_{\xi_{0}}\in\mathcal{M}_{+}\left(\Xi\right)$
such that $L(x^{*},\,\varLambda)>L(x^{*},\,\varLambda^{*})$ which
contradicts the fact that $L(x^{*},\,\varLambda)\leq L(x^{*},\,\varLambda^{*})$
for all $\varLambda\in\mathcal{M}_{+}\left(\Xi\right)$. If we take
$\varLambda=0$ in $L(x^{*},\,\varLambda)\leq L(x^{*},\,\varLambda^{*})$,
then we see that $\left\langle \varLambda^{*},\,G(x^{*})\right\rangle \geq0$
must hold. Since $\varLambda^{*}\geq0$ and $G(x^{*})\leq0$, must
in fact have $\left\langle \varLambda^{*},\,G(x^{*})\right\rangle =0$.
It follows that

\[
f(x^{*})\leq f(x)+\left\langle \varLambda^{*},\,G(x)\right\rangle ,\;\forall x\in\mathcal{X},
\]
which demonstrates that $x^{*}$ is optimal for Problem $\mathbb{P}$.
Indeed, if $x$ is feasible for Problem $\mathbb{P}$, then $\left\langle \varLambda^{*},\,G(x)\right\rangle \leq0$,
and therefore $f(x^{*})\leq f(x)$.

We now prove that $\varLambda^{*}$ is dual optimal. From the condition
that $L(x^{*},\,\varLambda^{*})\leq L(x,\,\varLambda^{*})$ for all
$x\in\mathcal{X}$, we must have that $d(\varLambda^{*})=L(x^{*},\,\varLambda^{*})$.
Next, we note $d(\varLambda)\leq L(x^{*},\,\varLambda)\leq L(x^{*},\,\varLambda^{*})$
for all $\varLambda\in\mathcal{M}_{+}\left(\Xi\right)$. Thus, $d(\varLambda)\leq d(\varLambda^{*})$
for all $\varLambda\in\mathcal{M}_{+}\left(\Xi\right)$, and so $\varLambda^{*}$
is optimal for Problem $\mathbb{D}$.
\end{proof}
Next, we prove that the set of dual optimal solutions lies within
a bounded set. 
\begin{proof}[Proof of Theorem \ref{dualupperbound}]
 By the definition of the dual functional $d(\cdot)$, for any dual
optimal solution $\varLambda^{\ast}\in\text{sol}\left(\mathbb{D}\right)$,
we have $f(\tilde{x})+\left\langle \varLambda^{*},\,G(\tilde{x})\right\rangle \geq d\left(\varLambda^{*}\right)\geq\mathrm{val}(\mathbb{D})$
(where $\tilde{x}$ satisfies the Slater condition). This inequality
implies that $-\left\langle \varLambda^{*},\,G(\tilde{x})\right\rangle \leq f(\tilde{x})-\mathrm{val}(\mathbb{D})$.
Because $\varLambda^{*}\geq0$ and $\left[G(\tilde{x})\right](\xi)=g(\tilde{x},\xi)<0$
for all $\xi\in\Xi$, it must be that

\[
\left\Vert \text{\ensuremath{\varLambda^{\ast}}}\right\Vert _{TV}\leq\frac{f(\tilde{x})-\mathrm{val}(\mathbb{D})}{\min_{\xi\in\Xi}\left\{ -g(\tilde{x},\xi)\right\} }=\frac{1}{\alpha}(f(\tilde{x})-\mathrm{val}(\mathbb{D})).
\]
\end{proof}

\subsection{\label{subsec:Properties-of B}The Prox Function}

In this subsection we develop the key properties of our new prox function
$B$. For this subsection, we define
\[
\mathcal{M}^{cd}\left(\Xi\right)\triangleq\left\{ \varLambda\in\mathcal{M}_{+}\left(\Xi\right):\,\varLambda\ll\varLambda_{u}\;\text{and}\;\varLambda\;\text{has a continuous density}\right\} .
\]

\paragraph*{Basic properties of $B$}

First we prove the preliminary statements in Lemma \ref{lem:prox_preliminary}.
\begin{proof}[Proof of Lemma \ref{lem:prox_preliminary}]
 (i) This is obvious. 

(ii) From direct calculation, we see $\frac{\partial^{2}H(\rho,\,\rho')}{\partial\rho^{2}}=\frac{1}{\rho}$,
which implies that $H(\rho,\,\rho')$ is convex in $\rho>0$. Further,
$\frac{\partial H(\rho,\,\rho')}{\partial\rho}=\log(\frac{\rho}{\rho'})$,
which implies that $H(\rho,\,\rho')$ is decreasing in $\rho$ when
$\rho\in(0,\,\rho')$ since here $\log(\frac{\rho}{\rho'})<0$, and
it is increasing in $\rho$ when $\rho>\rho'$ since here $\log(\frac{\rho}{\rho'})>0$.
Also, notice that $H(\rho',\,\rho')=0$. Therefore, $H(\rho,\,\rho')\geq0$,
and $H(\rho,\,\rho')=0$ if and only if $\rho=\rho'$.
\end{proof}
Next, we provide the proof of Theorem \ref{propertiesofB}.
\begin{proof}[Proof of Theorem \ref{propertiesofB}]
 (i) We first consider the case that $\varLambda\ll\varGamma$,

\[
\begin{array}{rcl}
B(\varLambda,\,\varGamma) & = & \intop_{\Xi}\log(\frac{\varLambda(\xi)}{\varGamma(\xi)})\varLambda(d\xi)-\left(\intop_{\Xi}\varLambda(d\xi)-\intop_{\Xi}\varGamma(d\xi)\right)\\
 & = & \intop_{\Xi}\left(\log(\frac{\phi(\varLambda)(\xi)}{\phi(\varGamma)(\xi)})+\log(\frac{\rho(\varLambda)}{\rho(\varGamma)})\right)\rho(\varLambda)\phi(\varLambda)(d\xi)-(\rho(\varLambda)-\rho(\varGamma))\\
 & = & \rho(\varLambda)\intop_{\Xi}\log(\frac{\phi(\varLambda)(\xi)}{\phi(\varGamma)(\xi)})\phi(\varLambda)(d\xi)+\rho(\varLambda)\log(\frac{\rho(\varLambda)}{\rho(\varGamma)})-\rho(\varLambda)+\rho(\varGamma)\\
 & = & \rho(\varLambda)\,D(\phi(\varLambda),\,\phi(\varGamma))+H(\rho(\varLambda),\,\rho(\varGamma)),
\end{array}
\]
where the second equality follows from 

\[
\log(\frac{\varLambda(\xi)}{\varGamma(\xi)})=\log(\frac{\phi(\varLambda)(\xi)}{\phi(\varGamma)(\xi)})+\log(\frac{\rho(\varLambda)}{\rho(\varGamma)}),
\]
and third equality holds because $\intop_{\Xi}\phi(\varLambda)(d\xi)=1$. 

If $\varLambda\not\ll\varGamma$, then $\phi(\varLambda)\not\ll\phi(\varGamma)$,
and thus we have $B(\varLambda,\,\varGamma)=\rho(\varLambda)\,D(\phi(\varLambda),\,\phi(\varGamma))+H(\rho(\varLambda),\,\rho(\varGamma))=+\infty$.

(ii) First, the KL divergence for any two probability distributions
is non-negative and it equals zero if and only if the two probability
distributions are the same (see \cite{bishop2006pattern}). Second,
$H(\rho(\varLambda),\,\rho(\varGamma))\geq0$, and it equals zero
if and only if $\rho(\varLambda)=\rho(\varGamma)$ from Lemma \ref{lem:prox_preliminary}.
Based on these two observations, we may invoke Theorem \ref{propertiesofB}(i)
and arrive at the conclusion.

(iii) If $\phi,\varphi\in\mathcal{P}\left(\Xi\right)$ are both densities,
then 

\[
B(\phi,\,\varphi)=\begin{cases}
\intop_{\Xi}\log(\frac{\phi(\xi)}{\varphi(\xi)})\phi(d\xi), & \phi\ll\varphi\text{ and }\intop_{\Xi}\left|\log(\frac{\phi(\xi)}{\varphi(\xi)})\right|\phi(d\xi)<\infty,\\
+\infty, & \text{otherwise,}
\end{cases}
\]
recovers the Kullback-Leibler divergence $D(\phi,\,\varphi)$ (see
\cite{bishop2006pattern}).

(iv) For $\mu\in[0,\,1]$, $\varLambda_{1}\ll\varGamma$ and $\varLambda_{2}\ll\varGamma$,
we have

\[
\begin{array}{rcl}
 &  & B(\mu\varLambda_{1}+(1-\mu)\varLambda_{2},\,\varGamma)\\
 & = & \intop_{\Xi}\log(\frac{\mu\varLambda_{1}(\xi)+(1-\mu)\varLambda_{2}(\xi)}{\varGamma(\xi)})(\mu\varLambda_{1}+(1-\mu)\varLambda_{2})(d\xi)-\left(\intop_{\Xi}(\mu\varLambda_{1}+(1-\mu)\varLambda_{2})(d\xi)-\intop_{\Xi}\varGamma(d\xi)\right)\\
 & \leq & \mu\intop_{\Xi}\log(\frac{\varLambda_{1}(\xi)}{\varGamma(\xi)})\varLambda_{1}(d\xi)+(1-\mu)\intop_{\Xi}\log(\frac{\varLambda_{2}(\xi)}{\varGamma(\xi)})\varLambda_{2}(d\xi)\\
 &  & -\mu\left(\intop_{\Xi}\varLambda_{1}(d\xi)-\intop_{\Xi}\varGamma(d\xi)\right)-(1-\mu)\left(\intop_{\Xi}\varLambda_{2}(d\xi)-\intop_{\Xi}\varGamma(d\xi)\right)\\
 & = & \mu B(\varLambda_{1},\,\varGamma)+(1-\mu)B(\varLambda_{2},\,\varGamma),
\end{array}
\]
where the inequality follows from convexity of $x\log x$ in $x>0$,
i.e., $(\mu x_{1}+(1-\mu)x_{2})\log(\mu x_{1}+(1-\mu)x_{2})\leq\mu x_{1}\log x_{1}+(1-\mu)x_{2}\log x_{2}$.

If either $\varLambda_{1}$ or $\varLambda_{2}$ is not absolutely
continuous with respect to $\varGamma$ (without loss of generality,
assume that $\varLambda_{1}$ is not absolutely continuous with respect
to $\varGamma$), then we have for $\mu\in(0,\,1)$, 

\[
B(\mu\varLambda_{1}+(1-\mu)\varLambda_{2},\,\varGamma)\leq\mu B(\varLambda_{1},\,\varGamma)+(1-\mu)B(\varLambda_{2},\,\varGamma)=+\infty.
\]

(v) Introduce the function $h(x)=x\log x-x+1,\,x\geq0$, where $0\log0\triangleq0$.
Notice that $h(x)=H(x,1)$ and thus from Lemma \ref{lem:prox_preliminary}
we have $h(x)\geq0,$ for any $x\geq0$. Next, we show 

\begin{equation}
\left(\frac{4}{3}+\frac{2}{3}x\right)h(x)\geq(x-1)^{2},\,x\geq0.\label{keyinPinsker}
\end{equation}
Inequality (\ref{keyinPinsker}) obviously holds for $x=0$. For $x>0$,
the function $Q(x)\triangleq\left(\frac{4}{3}+\frac{2}{3}x\right)h(x)-(x-1)^{2}$
satisfies $Q(1)=Q'(1)=0,\,Q''(x)=\frac{4h(x)}{3x}\geq0$. Thus, from
the Taylor's expansion, we have that for any $x>0$, there exists
$z$ which is between $1$ and $x$, such that 

\[
Q(x)=Q(1)+Q'(1)(x-1)+Q''(z)(x-1)^{2}\geq0,
\]
which proves (\ref{keyinPinsker}). 

For any $\varLambda,\,\varGamma\in\mathcal{M}_{+}\left(\Xi\right)$,
there always exists $\varPhi$ such that $\varLambda\ll\varPhi$ and
$\varGamma\ll\varPhi$ (for example, take $\varPhi=\varLambda+\varGamma$).
Abusing notation, we define the two Radon-Nikodym derivatives $\varLambda(\xi)=\varLambda(d\xi)/\varPhi(d\xi)$
and $\varGamma(\xi)=\varGamma(d\xi)/\varPhi(d\xi)$. If $\varLambda\ll\varGamma$,
then

\[
\begin{array}{rcl}
\left\Vert \varLambda-\varGamma\right\Vert _{TV}^{2} & = & \left(\intop_{\Xi}\left|\varLambda(\xi)-\varGamma(\xi)\right|\varPhi(d\xi)\right)^{2}=\left(\intop_{\Xi}\varGamma(\xi)\left|\frac{\varLambda(\xi)}{\varGamma(\xi)}-1\right|\varPhi(d\xi)\right)^{2}\\
 & \leq & \left(\intop_{\Xi}\varGamma(\xi)\sqrt{\left(\frac{4}{3}+\frac{2}{3}\frac{\varLambda(\xi)}{\varGamma(\xi)}\right)h(\frac{\varLambda(\xi)}{\varGamma(\xi)})}\varPhi(d\xi)\right)^{2}\\
 & \leq & \left(\intop_{\Xi}(\frac{4}{3}\varGamma(\xi)+\frac{2}{3}\varLambda(\xi))\varPhi(d\xi)\right)\left(\intop_{\Xi}\varGamma(\xi)h(\frac{\varLambda(\xi)}{\varGamma(\xi)})\varPhi(d\xi)\right)\\
 & \leq & 2\rho B(\varLambda,\,\varGamma),
\end{array}
\]
where the first inequality follows from (\ref{keyinPinsker}), the
second follows from the Cauchy-Schwarz inequality, and the last follows
because $\left\Vert \varLambda\right\Vert _{TV}\leq\rho$ and $\left\Vert \varGamma\right\Vert _{TV}\leq\rho$.
If $\varLambda\not\ll\varGamma$, the inequality is straightforward
since $B(\varLambda,\,\varGamma)=+\infty$. 
\end{proof}
~

\paragraph*{Differentiability of $B$}

We begin by recalling the definition of the Gateaux differential. 
\begin{defn}
\cite[Section 7.2]{luenberger1997optimization} Let $X$ be a vector
space, and $T$ be a functional defined on a domain $D\subset X$
and having range in $(-\infty,\infty]$. Let $x\in D$ and let $h$
be an arbitrary vector in $X.$ If the limit 

\begin{equation}
\delta T(x;\,h)\triangleq\lim_{\alpha\rightarrow0}\frac{1}{\alpha}[T(x+\alpha h)-T(x)]\label{Gateauxdiff}
\end{equation}
exists, it is called the Gateaux differential of $T$ at $x$ with
increment $h$. If the limit (\ref{Gateauxdiff}) exists for each
$h\in X$, the transformation $T$ is said to be Gateaux differentiable
at $x$. 
\end{defn}

The next lemma says that $B\left(\varLambda,\,\varGamma\right)$ is
a Bregman divergence, i.e., it is the difference between the value
of $B(\cdot\,,\varLambda_{u})$ at $\varLambda$ and the first order
expansion of $B(\cdot\,,\varLambda_{u})$ around $\varGamma$ evaluated
at $\varLambda$. Lemma \ref{dualupdatesolution} shows that all the
dual updates lie in $\mathcal{M}^{cd}\left(\Xi\right)$, so we only
consider the domain $\mathcal{M}^{cd}\left(\Xi\right)$. 
\begin{lem}
\label{Bregmanview}The Gateaux differential of $B\left(\cdot,\,\varLambda_{u}\right)\text{ : }\mathcal{M}^{cd}\left(\Xi\right)\rightarrow\mathbb{R}$
at $\varLambda$ with increment $h\in\mathcal{M}^{cd}\left(\Xi\right)$
is 
\[
\delta B(\varLambda;\,h)=\int_{\Xi}\log\left(\frac{\varLambda(\xi)\mathrm{vol}\left(\Xi\right)}{\rho_{0}}\right)h(\xi)d\xi.
\]
Furthermore, $B\left(\varLambda,\,\varGamma\right)=B\left(\varLambda,\,\varLambda_{u}\right)-B\left(\varGamma,\,\varLambda_{u}\right)-\delta B(\varGamma;\,\varLambda-\varGamma)$.
\end{lem}

\begin{proof}
We can rewrite $B\left(\varLambda,\,\varLambda_{u}\right)=\intop_{\Xi}c(\varLambda(\xi))d\xi$,
where $c(\lambda)\triangleq\log(\frac{\lambda\text{vol}\left(\Xi\right)}{\rho_{0}})\lambda-(\lambda-\rho_{0}\text{vol}\left(\Xi\right)^{-1})$,
$\lambda\geq0$. Notice that $c'(\lambda)=\log(\lambda\,\text{vol}\left(\Xi\right)/\rho_{0})$,
which is continuous in $\lambda$. From \cite[Section 7.2, Example 2]{luenberger1997optimization},
we see that the Gateaux differential of $B\left(\varLambda,\,\varLambda_{u}\right)$
at $\varLambda$ with increment $h$ is $\delta B(\varLambda;\,h)=\intop_{\Xi}\log\left(\varLambda(\xi)\text{vol}\left(\Xi\right)/\rho_{0}\right)h(\xi)d\xi$.
Then, it can be directly verified that 

\[
\begin{array}{rcl}
B\left(\varLambda,\,\varGamma\right) & = & B\left(\varLambda,\,\varLambda_{u}\right)-B\left(\varGamma,\,\varLambda_{u}\right)-\intop_{\Xi}\log\left(\varGamma(\xi)\text{vol}\left(\Xi\right)/\rho_{0}\right)(\varLambda(\xi)-\varGamma(\xi))d\xi\\
 & = & B\left(\varLambda,\,\varLambda_{u}\right)-B\left(\varGamma,\,\varLambda_{u}\right)-\delta B(\varGamma;\,\varLambda-\varGamma).
\end{array}
\]
\end{proof}
\begin{rem}
The Gateaux differential of $\varLambda\rightarrow B\left(\varLambda,\,\varLambda_{u}\right)$
does not exist on $\mathcal{M}_{+}\left(\Xi\right)$ since the limit
and the integral may not be interchangable. Lemma \ref{Bregmanview}
will be used in the proof of Lemma \ref{mu-mut} (which analyzes the
dual iterates for the regularized saddle-point problem). 
\end{rem}

\subsection{\label{subsec:Results-used-in Fan}The Regularized Saddle-Point Problem}

We closely examine the properties of Problem $\mathbb{SP}_{\kappa}$
in this subsection.

\paragraph*{Existence of saddle-points}

To proceed, we first prove Theorem \ref{existenceofSP_R} which guarantees
the existence of a saddle-point of Problem $\mathbb{SP}_{\kappa}$
by verifying the conditions of Fan's minimax theorem \cite[Theorem 1]{fan1953minimax}.

We need the following topological results. We equip the space $\mathcal{M}\left(\Xi\right)$
with the weak-star topology. In this topology, a base at $\varLambda\in\mathcal{M}\left(\Xi\right)$
is given by sets of the form 
\[
\left\{ \varLambda'\in\mathcal{M}\left(\Xi\right):\,\left|\left\langle \varLambda',\,f_{i}\right\rangle -\left\langle \varLambda,\,f_{i}\right\rangle \right|<\epsilon,\,i=1,2,\ldots,n\right\} 
\]
where $\left\{ f_{1},\,f_{2},\ldots,\,f_{n}\right\} $ is a finite
subset of $\mathcal{C}\left(\Xi\right)$.
\begin{lem}
\label{Hausdorff_compact_uppersemicont}(i) The weak-star topology
in $\mathcal{M}\left(\Xi\right)$ is Hausdorff. 

(ii) The set $\mathcal{V}$ is compact in the weak-star topology.

(iii) For each $x\in\mathcal{X}$, the mapping $\varLambda\rightarrow\left\langle \varLambda,\,G(x)\right\rangle -\kappa\,B\left(\varLambda,\,\varLambda_{u}\right)$
is upper semi-continuous with respect to the weak-star topology on
$\mathcal{V}$.
\end{lem}

\begin{proof}
(i) Let $\varLambda_{1}$ and $\varLambda_{2}$ be any two different
measures in $\mathcal{M}_{+}\left(\Xi\right)$, then there exists
$f\in\mathcal{C}\left(\Xi\right)$, such that $\intop_{\Xi}f(\xi)\varLambda_{1}(d\xi)\neq\intop_{\Xi}f(\xi)\varLambda_{2}(d\xi)$
and so $l\triangleq\left|\intop_{\Xi}f(\xi)\varLambda_{1}(d\xi)-\intop_{\Xi}f(\xi)\varLambda_{2}(d\xi)\right|\ne0$.
Further, we define 

\[
O(\varLambda_{i},\,l/3,\,f)\triangleq\left\{ \varLambda\in\mathcal{M}_{+}\left(\Xi\right):\,\left|\intop_{\Xi}f(\xi)\varLambda(d\xi)-\intop_{\Xi}f(\xi)\varLambda_{i}(d\xi)\right|<l/3\right\} ,\,i=1,2,
\]
which are two disjoint open sets in the weak-star topology in $\mathcal{M}_{+}\left(\Xi\right)$.
Clearly $\varLambda_{1}\in O(\varLambda_{1},\,l/3,\,f)$ and $\varLambda_{2}\in O(\varLambda_{2},\,l/3,\,f)$.
Therefore, the weak-star topology in $\mathcal{M}_{+}\left(\Xi\right)$
is Hausdorff.

(ii) Let $\left\{ \varLambda_{n}\right\} _{n\geq1}\subset\mathcal{V}$
be a sequence of non-negative measures on the compact set $\Xi$ (Assumption\textbf{\textcolor{black}{{}
A3}}). By Alaoglu\textquoteright s theorem and the Riesz representation
theorem, there exists a subsequence of $\left\{ \varLambda_{n}\right\} _{n\geq1}$
converging in the weak-star topology to a linear functional $F$ on
$\mathcal{C}\left(\Xi\right)$ with $\left\Vert F\right\Vert _{\mathcal{C}\left(\Xi\right)^{\ast}}\leq\overline{\rho}(\theta)$.
Further, this $F$ is non-negative and $F(e)\leq\overline{\rho}(\theta)$,
where $e\in\mathcal{C}\left(\Xi\right)$ is the constant function
everywhere equal to one on $\Xi$. It follows that $F$ corresponds
to a non-negative measure in $\mathcal{V}$. Hence, for any sequence
of measures in $\mathcal{V}$, there exists a subsequence converging
to a measure in $\mathcal{V}$ in the weak-star topology. 

(iii) Let $\left\{ \varLambda_{n}\right\} _{n\geq1}\subset\mathcal{V}$
converge to $\varLambda$ in the weak-star topology in $\mathcal{V}$.
Since $G(x)\in\mathcal{C}\left(\Xi\right)$ from Assumption \textbf{\textcolor{black}{A5}},
the mapping $\varLambda\rightarrow\left\langle \varLambda,\,G(x)\right\rangle $
is continuous with respect to the weak-star topology in $\mathcal{V}$
according to the definition of weak-star topology. In addition, $\Xi$
is compact from Assumptions \textbf{\textcolor{black}{A3}}. Noting
that $B\left(\varLambda,\,\varLambda_{u}\right)=\intop_{\Xi}c(\varLambda(\xi))d\xi$
where $c(\lambda)=\log(\lambda\text{vol}\left(\Xi\right)/\rho_{0})\lambda-(\lambda-\rho_{0}\text{vol}\left(\Xi\right)^{-1})$,
$\lambda\geq0$, and that $c(\lambda)$ is convex and continuous in
$\lambda$, we establish that the mapping $\varLambda\rightarrow B\left(\varLambda,\,\varLambda_{u}\right)$
is lower semi-continuous with respect to the weak-star topology in
$\mathcal{V}$ by invoking \cite[Theorem 5.27]{fonseca2007modern}.
Since $\kappa\in(0,1]$, and the negative of a lower semi-continuous
mapping is upper semi-continuous, we arrive at the desired conclusion.
\end{proof}
Next, we provide the proof for Theorem \ref{existenceofSP_R} by verifying
the three conditions of Fan's minimax theorem.
\begin{proof}[Proof of Theorem \ref{existenceofSP_R}]
 We check the following three conditions of Fan's minimax theorem
\cite[Theorem 1]{fan1953minimax}:

(i) From Assumption \textbf{A1}, $\mathcal{X}\subset\mathbb{R}^{m}$
is compact and Hausdorff. From Lemmas \ref{Hausdorff_compact_uppersemicont}(i)
and \ref{Hausdorff_compact_uppersemicont}(ii), it follows that $\mathcal{V}$
is Hausdorff and weak-star compact.

(ii) Now, $L_{\kappa}(x,\,\varLambda)$ is a real-valued function
defined on $\mathcal{X}\times\mathcal{V}$. For every $\varLambda\in\mathcal{V}$,
$L_{\kappa}(x,\,\varLambda)$ is continuous on $\mathcal{X}$ due
to Assumptions\textbf{ A2} and \textbf{A4}. From Lemma \ref{Hausdorff_compact_uppersemicont}(iii),
for every $x\in\mathcal{X}$, $L_{\kappa}(x,\,\varLambda)$ is upper
semi-continuous with respect to the weak-star topology on $\mathcal{V}$.

(iii) Next, notice that since $B\left(\varLambda,\,\varLambda_{u}\right)$
is convex in $\varLambda$ by Theorem \ref{propertiesofB}(iv) and
$\left\langle \varLambda,\,G(x)\right\rangle $ is linear in $\varLambda$,
$\left\langle \varLambda,\,G(x)\right\rangle -\kappa\,B\left(\varLambda,\,\varLambda_{u}\right)$
must be concave in $\varLambda$. Additionally, $f\left(x\right)+\left\langle \varLambda,\,G(x)\right\rangle $
is convex in $x$.
\end{proof}

\paragraph*{Approximation error}

To guarantee that Problem $\mathbb{SP}_{\kappa}$ closely approximates
the original Problem $\mathbb{SP}$, we want to show that the gap
between the value of the inner maximization over $\mathcal{V}$ of
the regularized Lagrangian and the value of the inner maximization
over $\mathcal{V}$ of the original Lagrangian can be made arbitrarily
small. 

Notice that $\mathcal{\mathcal{P}}(\Xi)$ is compact in the weak-star
topology since $\Xi$ is compact \cite[Theorem 15.11]{aliprantisinfinite},
and $\phi\rightarrow\mathbb{E}_{\xi\sim\phi}\left[g\left(x,\,\xi\right)\right]-\kappa\,D\left(\phi,\,\phi_{u}\right)$
is upper semi-continuous in $\phi\in\mathcal{\mathcal{P}}(\Xi)$ with
respect to the weak-star topology by the same reasoning as Lemma \ref{Hausdorff_compact_uppersemicont}(iii).
Therefore, the maximizer of $\mathbb{E}_{\xi\sim\phi}\left[g\left(x,\,\xi\right)\right]-\kappa\,D\left(\phi,\,\phi_{u}\right)$
is attained in $\phi\in\mathcal{\mathcal{P}}(\Xi)$, and we may denote
it as

\[
\phi_{\kappa}(x)\in\arg\max_{\phi\in\mathcal{\mathcal{P}}(\Xi)}\left\{ \mathbb{E}_{\xi\sim\phi}\left[g\left(x,\,\xi\right)\right]-\kappa\,D\left(\phi,\,\phi_{u}\right)\right\} ,
\]
for all fixed $x\in\mathcal{X}$.

The following result verifies that the inner maximization of Problem
$\mathbb{SP}_{\kappa}$ can be transformed into a two-stage optimization
problem: in the first stage, we optimize over densities $\phi\in\mathcal{\mathcal{P}}(\Xi)$;
in the second stage, we optimize over scalars $\rho\in[0,\,\bar{\rho}(\theta)]$. 
\begin{lem}
\label{Vrhophi} For any $\kappa\in(0,1]$ and $x\in\mathcal{X}$,
we have 
\begin{equation}
\max_{\varLambda\in\mathcal{V}}\left\{ \left\langle \varLambda,\,G(x)\right\rangle -\kappa\,B\left(\varLambda,\,\varLambda_{u}\right)\right\} =\max_{\rho\in[0,\,\bar{\rho}(\theta)]}\left\{ \rho\max_{\phi\in\mathcal{\mathcal{P}}(\Xi)}\left[\mathbb{E}_{\xi\sim\phi}\left[g\left(x,\,\xi\right)\right]-\kappa\,D\left(\phi,\,\phi_{u}\right)\right]-\kappa H(\rho,\rho_{0})\right\} .\label{twostageforinnermaxofSP_R}
\end{equation}
\end{lem}

\begin{proof}
From Theorem \ref{propertiesofB}(i), we directly obtain

\[
\left\langle \varLambda,\,G(x)\right\rangle -\kappa\,B\left(\varLambda,\,\varLambda_{u}\right)=\rho(\varLambda)\left[\mathbb{E}_{\xi\sim\phi(\varLambda)}\left[g\left(x,\,\xi\right)\right]-\kappa\,D\left(\phi(\varLambda),\,\phi_{u}\right)\right]-\kappa H(\rho(\varLambda),\rho_{0}).
\]
Next, there is a one-to-one correspondence between $\varLambda\in\mathcal{V}$
and $(\rho(\varLambda),\,\phi(\varLambda))\in[0,\,\bar{\rho}(\theta)]\times\mathcal{P}(\Xi)$
by Lemma \ref{lem:prox_preliminary}(i). We observe that the inner
maximization in the right hand side of (\ref{twostageforinnermaxofSP_R})
is decoupled from the outer one, so we arrive at the conclusion.
\end{proof}
The next result shows that the inner optimization problem in (\ref{twostageforinnermaxofSP_R})
has a closed form solution by the calculus of variations.
\begin{lem}
\label{maxoverphi} For any $\kappa\in(0,1]$ and $x\in\mathcal{X}$,
we have 

\[
\left[\phi_{\kappa}(x)\right](\xi)=\frac{\exp\left(g(x,\xi)/\kappa\right)}{\intop_{\Xi}\exp\left(g(x,\xi)/\kappa\right)d\xi},\,\forall\xi\in\Xi.
\]
Furthermore, 

\begin{equation}
\max_{\phi\in\mathcal{\mathcal{P}}(\Xi)}\left\{ \mathbb{E}_{\xi\sim\phi}\left[g\left(x,\,\xi\right)\right]-\kappa\,D\left(\phi,\,\phi_{u}\right)\right\} =\kappa\log\left(\intop_{\Xi}\exp\left(g(x,\xi)/\kappa\right)d\xi\right)-\kappa\log(\mathrm{vol}\left(\Xi\right)).\label{innermaxoverprob}
\end{equation}
\end{lem}

\begin{proof}
We note that $\max_{\phi\in\mathcal{\mathcal{P}}(\Xi)}\left\{ \mathbb{E}_{\xi\sim\phi}\left[g\left(x,\,\xi\right)\right]-\kappa\,D\left(\phi,\,\phi_{u}\right)\right\} $
is a constrained calculus of variations problem:

\[
\begin{array}{rcl}
\max_{\phi\in\mathcal{M}_{+}\left(\Xi\right)} &  & \intop_{\Xi}g\left(x,\,\xi\right)\phi(\xi)-\kappa\log\left(\phi(\xi)\mathrm{vol}\left(\Xi\right)\right)\phi(\xi)d\xi\\
s.t. &  & \intop_{\Xi}\phi(\xi)d\xi=1.
\end{array}
\]
Using Euler's equation in the calculus of variations (see \cite[Section 7.5]{luenberger1997optimization}),
we obtain after simplification, 

\begin{equation}
\kappa\,\log\left(\phi(\xi)\right)=g\left(x,\,\xi\right)+C,\,\forall\xi\in\Xi,\label{Euler-1}
\end{equation}
where $C=\upsilon+\kappa\log(\mathrm{vol}\left(\Xi\right))-\kappa$
and $\upsilon\in\mathbb{R}$ is the Lagrange multiplier of the constraint
$\intop_{\Xi}\phi(\xi)d\xi=1$. From (\ref{Euler-1}) and the constraint
$\intop_{\Xi}\phi(\xi)d\xi=1$, we obtain 

\begin{equation}
\left[\phi_{\kappa}(x)\right](\xi)=\frac{\exp\left(g(x,\xi)/\kappa\right)}{\intop_{\Xi}\exp\left(g(x,\xi)/\kappa\right)d\xi},\,\forall\xi\in\Xi.\label{phi_Rx}
\end{equation}
Therefore, by replacing (\ref{phi_Rx}) in $\mathbb{E}_{\xi\sim\phi}\left[g\left(x,\,\xi\right)\right]-\kappa\,D\left(\phi,\,\phi_{u}\right)$,
we obtain (\ref{innermaxoverprob}) after simplification. 
\end{proof}
We need the following intermediate result to bound $\max_{\phi\in\mathcal{\mathcal{P}}(\Xi)}\left\{ \mathbb{E}_{\xi\sim\phi}\left[g\left(x,\,\xi\right)\right]-\kappa\,D\left(\phi,\,\phi_{u}\right)\right\} $,
and this bound plays an important role later in establishing Lemma
\ref{maxsupdiff}. Since $\Xi\subset\mathbb{R}^{d}$ is full dimensional,
there exist balls contained in $\Xi$. Recall that $B_{R_{\Xi}}(\tilde{\xi})$
is the largest Euclidean ball included in $\Xi\subset\mathbb{R}^{d}$
with radius $R_{\Xi}$ centered at $\tilde{\xi}$ and its volume is
$\frac{\pi^{d/2}}{\Gamma(d/2+1)}R_{\Xi}^{d}$. In this proof, we use
the Assumption \textbf{\textcolor{black}{A3}} that $\Xi$ is full
dimensional and convex. 
\begin{lem}
\label{importbound} For any fixed $\kappa\in(0,1]$ and $x\in\mathcal{X},$
we have 

\begin{equation}
\int_{\Xi}\exp\left(\frac{g(x,\xi)}{\kappa}\right)d\xi\geq\exp\left(\frac{1}{\kappa}\max_{\xi\in\Xi}g(x,\xi)\right)\exp\left(-L_{g,\,\Xi}(R_{\Xi}+D_{\Xi})\right)\frac{\pi^{d/2}}{\Gamma(d/2+1)}(\kappa R_{\Xi})^{d}.\label{keybound}
\end{equation}
\end{lem}

\begin{proof}
We have

\begin{equation}
\begin{array}{rcl}
\int_{\Xi}\exp\left(\frac{g(x,\xi)}{\kappa}\right)d\xi & = & \exp\left(\frac{1}{\kappa}\max_{\xi\in\Xi}g(x,\xi)\right)\int_{\Xi}\exp\left(-\frac{1}{\kappa}(\max_{\xi\in\Xi}g(x,\xi)-g(x,\xi))\right)d\xi\\
 & \geq & \exp\left(\frac{1}{\kappa}\max_{\xi\in\Xi}g(x,\xi)\right)\int_{\Xi}\exp\left(-\frac{L_{g,\,\Xi}}{\kappa}\left\Vert \xi^{\ast}(x)-\xi\right\Vert \right)d\xi,
\end{array}\label{keybound1}
\end{equation}
where $\xi^{\ast}(x)\in\arg\max_{\xi\in\Xi}g(x,\,\xi)$, and the last
inequality follows by Assumption \textbf{\textcolor{black}{A5}}.

Define $\xi_{\kappa}\triangleq\kappa\tilde{\xi}+(1-\kappa)\xi^{\ast}(x)$.
Since $\Xi$ is convex by Assumption \textbf{\textcolor{black}{A3}},
we deduce $B_{\kappa R_{\Xi}}(\xi_{\kappa})=(1-\kappa)\xi^{\ast}(x)+\kappa B_{R_{\Xi}}(\tilde{\xi})\subseteq\Xi$,
which implies that, for any $\xi\in B_{\kappa R_{\Xi}}(\xi_{\kappa}),$
there exists an element $\xi'\in B_{R_{\Xi}}(\tilde{\xi})$ such that
$\xi=\kappa\xi'+(1-\kappa)\xi^{\ast}(x)$. Then, for any $\xi\in B_{\kappa R_{\Xi}}(\xi_{\kappa}),$
we have $\left\Vert \xi^{\ast}(x)-\xi\right\Vert =\kappa\left\Vert \xi'-\xi^{\ast}(x)\right\Vert \leq\kappa\left(\left\Vert \xi'-\tilde{\xi}\right\Vert +\left\Vert \tilde{\xi}-\xi^{\ast}(x)\right\Vert \right)\leq\kappa(R_{\Xi}+D_{\Xi})$.
Therefore, we have

\begin{equation}
\begin{array}{rcl}
\int_{\Xi}\exp\left(-\frac{L_{g,\,\Xi}}{\kappa}\left\Vert \xi^{\ast}(x)-\xi\right\Vert \right)d\xi & \geq & \int_{B_{\kappa R_{\Xi}}(\xi_{\kappa})}\exp\left(-\frac{L_{g,\,\Xi}}{\kappa}\left\Vert \xi^{\ast}(x)-\xi\right\Vert \right)d\xi\\
 & \geq & \exp\left(-L_{g,\,\Xi}(R_{\Xi}+D_{\Xi})\right)\frac{\pi^{d/2}}{\Gamma(d/2+1)}(\kappa R_{\Xi})^{d}.
\end{array}\label{keybound2}
\end{equation}
Combining (\ref{keybound1}) and (\ref{keybound2}), we arrive at
the conclusion. 
\end{proof}
We are ready to give an explicit bound on the gap between the values
of the inner maximization of the regularized Lagrangian and the original
Lagrangian. 
\begin{lem}
\label{maxsupdiff} For any $\kappa\in(0,1]$ and $x\in\mathcal{X}$, 

\begin{equation}
\max_{\varLambda\in\mathcal{V}}\left\{ \left\langle \varLambda,\,G(x)\right\rangle -\kappa\,B\left(\varLambda,\,\varLambda_{u}\right)\right\} \geq\max_{\varLambda\in\mathcal{V}}\left\{ \left\langle \varLambda,\,G(x)\right\rangle \right\} +\bar{\rho}(\theta)\kappa\log(\kappa)d-\kappa\bar{C}\left(\theta\right).\label{maxdiff}
\end{equation}
\end{lem}

\begin{proof}
Applying (\ref{keybound}) to bound the term $\log\left(\intop_{\Xi}\exp\left(g(x,\xi)/\kappa\right)d\xi\right)$
in the right hand side of (\ref{innermaxoverprob}), we obtain 

\[
\begin{array}{rcl}
 &  & \max_{\phi\in\mathcal{\mathcal{P}}(\Xi)}\left\{ \mathbb{E}_{\xi\sim\phi}\left[g\left(x,\,\xi\right)\right]-\kappa\,D\left(\phi,\,\phi_{u}\right)\right\} \\
 & \geq & \kappa\log\left(\exp\left(\frac{1}{\kappa}\max_{\xi\in\Xi}g(x,\xi)\right)\exp\left(-L_{g,\,\Xi}(R_{\Xi}+D_{\Xi})\right)\frac{\pi^{d/2}}{\Gamma(d/2+1)}(\kappa R_{\Xi})^{d}\right)-\kappa\log(\text{vol}\left(\Xi\right))\\
 & = & \max_{\phi\in\mathcal{\mathcal{P}}(\Xi)}\mathbb{E}_{\xi\sim\phi}\left[g\left(x,\,\xi\right)\right]+\kappa\log(\kappa)d-\kappa L_{g,\,\Xi}(R_{\Xi}+D_{\Xi})+\kappa\log\left(r\right),
\end{array}
\]
where the last equality follows by $\max_{\xi\in\Xi}g(x,\xi)=\max_{\phi\in\mathcal{\mathcal{P}}(\Xi)}\mathbb{E}_{\xi\sim\phi}\left[g\left(x,\,\xi\right)\right]$.
It follows from Lemma \ref{Vrhophi} that

\[
\begin{array}{rll}
 & \max_{\varLambda\in\mathcal{V}} & \left\{ \left\langle \varLambda,\,G(x)\right\rangle -\kappa\,B\left(\varLambda,\,\varLambda_{u}\right)\right\} \\
\geq & \max_{\rho\in[0,\,\bar{\rho}(\theta)]} & \rho\left[\max_{\phi\in\mathcal{\mathcal{P}}(\Xi)}\mathbb{E}_{\xi\sim\phi}\left[g\left(x,\,\xi\right)\right]+\kappa\log(\kappa)d-\kappa L_{g,\,\Xi}(R_{\Xi}+D_{\Xi})+\kappa\log\left(r\right)\right]\\
 &  & -\kappa H(\rho,\rho_{0}).
\end{array}
\]
To complete the argument, we consider the following two cases: 

Case 1: If $\max_{\phi\in\mathcal{\mathcal{P}}(\Xi)}\mathbb{E}_{\xi\sim\phi}\left[g\left(x,\,\xi\right)\right]>0$,
then $\max_{\varLambda\in\mathcal{V}}\left\{ \left\langle \varLambda,\,G(x)\right\rangle \right\} =\bar{\rho}(\theta)\max_{\phi\in\mathcal{\mathcal{P}}(\Xi)}\mathbb{E}_{\xi\sim\phi}\left[g\left(x,\,\xi\right)\right]$,
and 

\[
\begin{array}{rcl}
 &  & \max_{\rho\in[0,\,\bar{\rho}(\theta)]}\rho\left[\max_{\phi\in\mathcal{\mathcal{P}}(\Xi)}\mathbb{E}_{\xi\sim\phi}\left[g\left(x,\,\xi\right)\right]+\kappa\log(\kappa)d-\kappa L_{g,\,\Xi}(R_{\Xi}+D_{\Xi})+\kappa\log\left(r\right)\right]-\kappa H(\rho,\rho_{0})\\
 & \geq & \bar{\rho}(\theta)\left[\max_{\phi\in\mathcal{\mathcal{P}}(\Xi)}\mathbb{E}_{\xi\sim\phi}\left[g\left(x,\,\xi\right)\right]+\kappa\log(\kappa)d-\kappa L_{g,\,\Xi}(R_{\Xi}+D_{\Xi})+\kappa\log\left(r\right)\right]-\kappa H(\bar{\rho}(\theta),\rho_{0})\\
 & \geq & \max_{\varLambda\in\mathcal{V}}\left\{ \left\langle \varLambda,\,G(x)\right\rangle \right\} +\bar{\rho}(\theta)\kappa\log(\kappa)d-\bar{\rho}(\theta)\kappa L_{g,\,\Xi}(R_{\Xi}+D_{\Xi})+\bar{\rho}(\theta)\kappa\log\left(r\right)-\kappa H_{\max},
\end{array}
\]
where the first inequality follows by the special selection $\rho=\bar{\rho}(\theta)$,
and the second inequality follows because $H_{\max}\geq H(\bar{\rho}(\theta),\rho_{0})$.

Case 2: If $\max_{\phi\in\mathcal{\mathcal{P}}(\Xi)}\mathbb{E}_{\xi\sim\phi}\left[g\left(x,\,\xi\right)\right]\leq0$,
then $\max_{\varLambda\in\mathcal{V}}\left\{ \left\langle \varLambda,\,G(x)\right\rangle \right\} =0$,
and 

\[
\begin{array}{rcl}
 &  & \max_{\rho\in[0,\,\bar{\rho}(\theta)]}\rho\left[\max_{\phi\in\mathcal{\mathcal{P}}(\Xi)}\mathbb{E}_{\xi\sim\phi}\left[g\left(x,\,\xi\right)\right]+\kappa\log(\kappa)d-\kappa L_{g,\,\Xi}(R_{\Xi}+D_{\Xi})+\kappa\log\left(r\right)\right]-\kappa H(\rho,\rho_{0})\\
 & \geq & -\kappa\rho_{0}\\
 & \geq & \max_{\varLambda\in\mathcal{V}}\left\{ \left\langle \varLambda,\,G(x)\right\rangle \right\} +\bar{\rho}(\theta)\kappa\log(\kappa)d-\bar{\rho}(\theta)\kappa L_{g,\,\Xi}(R_{\Xi}+D_{\Xi})+\bar{\rho}(\theta)\kappa\log\left(r\right)-\kappa H_{\max},
\end{array}
\]
where the first inequality follows by the special selection $\rho=0$
and $H(0,\rho_{0})=\rho_{0}$, and the second inequality follows because
$\bar{\rho}(\theta)\kappa\log(\kappa)d-\bar{\rho}(\theta)\kappa L_{g,\,\Xi}(R_{\Xi}+D_{\Xi})+\bar{\rho}(\theta)\kappa\log\left(r\right)\leq0$
and $H_{\max}\geq H(0,\rho_{0})$. The argument is now complete.
\end{proof}
We are in a position to claim that there is an upper bound on $B(\varLambda_{\kappa}(x),\,\varLambda_{u})$
(recall that $\varLambda_{\kappa}\left(x\right)\in\mathcal{V}$ is
an inner maximizer in Problem $\mathbb{SP}_{\kappa}$ for fixed $x\in\mathcal{X}$,
which is defined in (\ref{ReguSoluLambdaRx})). We define

\[
C'(\theta)\triangleq-\bar{\rho}(\theta)\log(\kappa)d+\bar{\rho}(\theta)L_{g,\,\xi}(R_{\Xi}+D_{\Xi})-\bar{\rho}(\theta)\log\left(r\right)+H_{\max}.
\]

\begin{lem}
\label{LemmaQB} For any $\kappa\in(0,1]$, $B(\varLambda_{\kappa}(x),\,\varLambda_{u})\leq C'(\theta)$
for all $x\in\mathcal{X}$.
\end{lem}

\begin{proof}
From the definition of $\varLambda_{\kappa}(x),$ we have that 

\[
\max_{\varLambda\in\mathcal{V}}\left\{ \left\langle \varLambda,\,G(x)\right\rangle -\kappa\,B\left(\varLambda,\,\varLambda_{u}\right)\right\} =\left\langle \varLambda_{\kappa}(x),\,G(x)\right\rangle -\kappa\,B\left(\varLambda_{\kappa}(x),\,\varLambda_{u}\right).
\]
Then, 

\[
\begin{array}{rcl}
\kappa B(\varLambda_{\kappa}(x),\,\varLambda_{u}) & = & -\max_{\varLambda\in\mathcal{V}}\left\{ \left\langle \varLambda,\,G(x)\right\rangle -\kappa\,B\left(\varLambda,\,\varLambda_{u}\right)\right\} +\left\langle \varLambda_{\kappa}(x),\,G(x)\right\rangle \\
 & \leq & -\max_{\varLambda\in\mathcal{V}}\left\{ \left\langle \varLambda,\,G(x)\right\rangle -\kappa\,B\left(\varLambda,\,\varLambda_{u}\right)\right\} +\max_{\varLambda\in\mathcal{V}}\left\{ \left\langle \varLambda,\,G(x)\right\rangle \right\} \\
 & \leq & -\bar{\rho}(\theta)\kappa\log(\kappa)d+\bar{\rho}(\theta)\kappa L_{g,\,\Xi}(R_{\Xi}+D_{\Xi})-\bar{\rho}(\theta)\kappa\log\left(r\right)+\kappa H_{\max},
\end{array}
\]
where the first inequality follows because $\max_{\varLambda\in\mathcal{V}}\left\{ \left\langle \varLambda,\,G(x)\right\rangle \right\} \geq\left\langle \varLambda_{\kappa}(x),\,G(x)\right\rangle $,
and the second inequality follows from Lemma \ref{maxsupdiff}. Since
$\kappa>0$ we obtain $B(\varLambda_{\kappa}(x),\,\varLambda_{u})\leq C'(\theta)$.
\end{proof}
\begin{rem}
\label{boundedB} For our convergence analysis to go through, the
prox function term must be bounded. Notice that the inner maximizer
over $\varLambda\in\mathcal{V}$ in Problem $\mathbb{SP}$ can be
a point measure, in which case $B(\varLambda,\,\varLambda_{u})$ would
be unbounded. From Lemma \ref{LemmaQB}, there is an upper bound on
$B(\varLambda_{\kappa}(x),\,\varLambda_{u})$ for all $x\in\mathcal{X}$.
Actually, this observation is the essential reason why we solve the
regularized saddle-point Problem $\mathbb{SP}_{\kappa}$ rather than
the original saddle-point Problem $\mathbb{SP}$. 
\end{rem}

\subsection{\label{subsec:Convex-concave analysis} Convex-Concave Analysis}

The goal of this subsection is to analyze the primal-dual iterates
for the regularized saddle-point problem, as well as the averaged
primal and dual solutions $\overline{x}_{K}$ and $\overline{\varLambda}_{K}$.
The main result of this subsection is a bound on the difference $L_{\kappa}(\overline{x}_{K},\,\varLambda)-L_{\kappa}(x,\,\overline{\varLambda}_{K})$,
which paves the way for the convergence proof of our inexact primal-dual
algorithm through appropriate choice of $x$ and $\varLambda$. 

The upcoming Lemmas \ref{mu-mut} and \ref{xt-x} form the basis for
our analysis of the averaged primal and dual solutions $\overline{x}_{K}$
and $\overline{\varLambda}_{K}$. 
\begin{lem}
\label{mu-mut} For any $\varLambda\in\mathcal{M}^{cd}\left(\Xi\right)\cap\mathcal{V}$
and for all $k=0,\,1,\ldots,\,K-1,$

\[
\begin{array}{rcl}
 &  & \left\langle \varLambda-\varLambda_{k},\,G(x_{k})\right\rangle -\kappa\,(B\left(\varLambda,\,\varLambda_{u}\right)-B\left(\varLambda_{k},\,\varLambda_{u}\right))\\
 & \leq & \frac{1}{\gamma}(B(\varLambda,\,\varLambda_{k})-B(\varLambda,\,\varLambda_{k+1}))+\kappa(B\left(\varLambda_{k},\,\varLambda_{u}\right)-B(\varLambda_{k+1},\varLambda_{u}))+\frac{\gamma\bar{\rho}(\theta)}{2}G_{\max}^{2}.
\end{array}
\]
\end{lem}

\begin{proof}
Recall that $\varLambda_{k+1}$ solves the convex optimization problem
(\ref{dualupdate}). From \cite[Section 7.4, Theorem 2]{luenberger1997optimization}
and Lemma \ref{Bregmanview}, the optimality of $\varLambda_{k+1}$
implies that for any $\varLambda\in\mathcal{M}^{cd}\left(\Xi\right)\cap\mathcal{V}$,

\[
\left\langle \gamma[-G(x_{k})+\kappa\log(\varLambda_{k+1}/\varLambda_{u})]+\log(\varLambda_{k+1}/\varLambda_{k}),\,\varLambda_{k+1}-\varLambda\right\rangle \leq0,
\]
 which is equivalent to 

\begin{eqnarray}
 &  & -\left\langle \log(\varLambda_{k+1}/\varLambda_{k}),\,\varLambda\right\rangle \nonumber \\
 & \leq & \gamma\left\langle G(x_{k}),\,\varLambda_{k+1}-\varLambda\right\rangle -\gamma\kappa\left\langle \log(\varLambda_{k+1}/\varLambda_{u}),\,\varLambda_{k+1}-\varLambda\right\rangle -\left\langle \log(\varLambda_{k+1}/\varLambda_{k}),\,\varLambda_{k+1}\right\rangle \nonumber \\
 & = & \gamma\left\langle G(x_{k}),\,\varLambda_{k+1}-\varLambda\right\rangle -\gamma\kappa[B(\varLambda_{k+1},\,\varLambda_{u})-B(\varLambda,\,\varLambda_{u})+B(\varLambda,\,\varLambda_{k+1})]\nonumber \\
 &  & -[B(\varLambda_{k+1},\,\varLambda_{k})+\left\langle \varLambda_{k+1}-\varLambda_{k},\,1\right\rangle ].\label{rearrange}
\end{eqnarray}
Moreover, we have

\begin{equation}
-\left\langle \log(\varLambda_{k+1}/\varLambda_{k}),\,\varLambda\right\rangle =B(\varLambda,\,\varLambda_{k+1})-B(\varLambda,\,\varLambda_{k})-\left\langle \varLambda_{k+1}-\varLambda_{k},\,1\right\rangle .\label{eq:D-D}
\end{equation}
From (\ref{rearrange}) and (\ref{eq:D-D}), by rearranging terms
we obtain

\begin{eqnarray}
 &  & \left\langle G(x_{k}),\,\varLambda-\varLambda_{k}\right\rangle -\kappa\,(B\left(\varLambda,\,\varLambda_{u}\right)-B\left(\varLambda_{k},\,\varLambda_{u}\right))\nonumber \\
 & \leq & \frac{1}{\gamma}B(\varLambda,\,\varLambda_{k})-(\kappa+\frac{1}{\gamma})B(\varLambda,\,\varLambda_{k+1})+\kappa(B\left(\varLambda_{k},\,\varLambda_{u}\right)-B(\varLambda_{k+1},\varLambda_{u}))\nonumber \\
 &  & -\frac{1}{\gamma}B(\varLambda_{k+1},\varLambda_{k})+\left\langle G(x_{k}),\,\varLambda_{k+1}-\varLambda_{k}\right\rangle \nonumber \\
 & \leq & \frac{1}{\gamma}B(\varLambda,\,\varLambda_{k})-(\kappa+\frac{1}{\gamma})B(\varLambda,\,\varLambda_{k+1})+\kappa(B\left(\varLambda_{k},\,\varLambda_{u}\right)-B(\varLambda_{k+1},\varLambda_{u}))\nonumber \\
 &  & -\frac{1}{2\gamma\bar{\rho}(\theta)}\left\Vert \varLambda_{k+1}-\varLambda_{k}\right\Vert _{TV}^{2}+\left\langle G(x_{k}),\,\varLambda_{k+1}-\varLambda_{k}\right\rangle \nonumber \\
 & \leq & \frac{1}{\gamma}B(\varLambda,\,\varLambda_{k})-(\kappa+\frac{1}{\gamma})B(\varLambda,\,\varLambda_{k+1})+\kappa(B\left(\varLambda_{k},\,\varLambda_{u}\right)-B(\varLambda_{k+1},\varLambda_{u}))+\frac{\gamma\bar{\rho}(\theta)}{2}\left\Vert G(x_{k})\right\Vert _{\mathcal{C}\left(\Xi\right)}^{2}\nonumber \\
 & \leq & \frac{1}{\gamma}(B(\varLambda,\,\varLambda_{k})-B(\varLambda,\,\varLambda_{k+1}))+\kappa(B\left(\varLambda_{k},\,\varLambda_{u}\right)-B(\varLambda_{k+1},\varLambda_{u}))+\frac{\gamma\bar{\rho}(\theta)}{2}G_{\max}^{2}\text{,}\label{key1}
\end{eqnarray}
where the second inequality holds by $B(\varLambda_{k+1},\varLambda_{k})\geq\frac{1}{2\bar{\rho}(\theta)}\left\Vert \varLambda_{k+1}-\varLambda_{k}\right\Vert _{TV}^{2}$
from Theorem \ref{propertiesofB}(iii), the third inequality holds
by Young's inequality, and the last inequality holds because $\kappa B(\varLambda,\,\varLambda_{k+1})$
is non-negative. 
\end{proof}
To continue, we define an auxiliary sequence denoted as $\tilde{x}_{k}\in\mathcal{X}$
for $k\geq0$. We initialize $\tilde{x}_{0}=x_{0}$, and then for
all $k\geq0$ we set

\begin{equation}
\tilde{x}_{k+1}=\arg\min_{x\in\mathcal{X}}\left[\gamma\langle x-\tilde{x}_{k},\,-\varepsilon_{k}\rangle+\frac{1}{2}\|x-\tilde{x}_{k}\|^{2}\right].\label{hatx}
\end{equation}
The optimality of $x_{k+1}$ in (\ref{inexactprimalupdate}) and the
optimality of $\tilde{x}_{k+1}$ in (\ref{hatx}) will be used together
to analyze the primal iterate $x_{k}$ in the next result. This auxiliary
sequence serves as an intermediate term which will only appear in
the following lemma.
\begin{lem}
\label{xt-x} For any $x\in\mathcal{X}$ and for all $k=0,\,1,\ldots,\,K-1$,
\end{lem}

\[
\begin{array}{rcl}
 &  & \left\langle x_{k}-x,\,\nabla f(x_{k})+\left\langle \varLambda_{k},\,\nabla G\left(x_{k}\right)\right\rangle \right\rangle \\
 & \leq & \frac{1}{2\gamma}\left(\left\Vert x-x_{k}\right\Vert ^{2}-\left\Vert x-x_{k+1}\right\Vert ^{2}\right)+\frac{1}{2\gamma}\left(\left\Vert x-\tilde{x}_{k}\right\Vert ^{2}-\left\Vert x-\tilde{x}_{k+1}\right\Vert ^{2}\right)\\
 &  & +\gamma(L_{f}+\bar{\rho}(\theta)L_{g,\,\mathcal{X}})^{2}+\frac{3\gamma}{2}\left\Vert \varepsilon_{k}\right\Vert ^{2}+\sqrt{D_{\mathcal{X}}}\left\Vert \varepsilon_{k}\right\Vert .
\end{array}
\]

\begin{proof}
Recall that $x_{k+1}$ is a minimizer of Problem (\ref{inexactprimalupdate}).
From \cite[Section 7.4, Theorem 2]{luenberger1997optimization}, the
optimality of $x_{k+1}$ implies that for any $x\in\mathcal{X}$ we
must have

\[
\left\langle \gamma\left(\nabla f(x_{k})+\left\langle \varLambda_{k},\,\nabla G\left(x_{k}\right)\right\rangle +\varepsilon_{k}\right)+x_{k+1}-x_{k},\,x_{k+1}-x\right\rangle \leq0.
\]
This inequality indicates that for any $x\in\mathcal{X}$,

\begin{eqnarray}
 &  & \left\langle x_{k}-x,\,\nabla f(x_{k})+\left\langle \varLambda_{k},\,\nabla G\left(x_{k}\right)\right\rangle +\varepsilon_{k}\right\rangle \nonumber \\
 & \leq & \left\langle x_{k}-x_{k+1},\,\nabla f(x_{k})+\left\langle \varLambda_{k},\,\nabla G\left(x_{k}\right)\right\rangle +\varepsilon_{k}\right\rangle -\frac{1}{\gamma}\left\langle x_{k+1}-x_{k},\,x_{k+1}-x\right\rangle \nonumber \\
 & = & \left\langle x_{k}-x_{k+1},\,\nabla f(x_{k})+\left\langle \varLambda_{k},\,\nabla G\left(x_{k}\right)\right\rangle +\varepsilon_{k}\right\rangle \nonumber \\
 &  & +\frac{1}{2\gamma}\left(\left\Vert x-x_{k}\right\Vert ^{2}-\left\Vert x-x_{k+1}\right\Vert ^{2}-\left\Vert x_{k+1}-x_{k}\right\Vert ^{2}\right)\nonumber \\
 & \leq & \frac{1}{2\gamma}\left(\left\Vert x-x_{k}\right\Vert ^{2}-\left\Vert x-x_{k+1}\right\Vert ^{2}\right)+\frac{\gamma}{2}\left\Vert \nabla f(x_{k})+\left\langle \varLambda_{k},\,\nabla G\left(x_{k}\right)\right\rangle +\varepsilon_{k}\right\Vert ^{2}\nonumber \\
 & \leq & \frac{1}{2\gamma}\left(\left\Vert x-x_{k}\right\Vert ^{2}-\left\Vert x-x_{k+1}\right\Vert ^{2}\right)+\gamma(L_{f}+\bar{\rho}(\theta)L_{g,\,\mathcal{X}})^{2}+\gamma\left\Vert \varepsilon_{k}\right\Vert ^{2},\label{telescope}
\end{eqnarray}
where the equality follows by $\left\Vert x-x_{k}\right\Vert ^{2}-\left\Vert x-x_{k+1}\right\Vert ^{2}-\left\Vert x_{k+1}-x_{k}\right\Vert ^{2}=-2\left\langle x_{k+1}-x_{k},\,x_{k+1}-x\right\rangle $,
the second inequality holds by Young's inequality, and the last inequality
follows from Assumptions \textbf{A2}, \textbf{A4}, and

\[
\frac{1}{2}\left\Vert \nabla f(x_{k})+\left\langle \varLambda_{k},\,\nabla G\left(x_{k}\right)\right\rangle +\varepsilon_{k}\right\Vert ^{2}\leq\left\Vert \nabla f(x_{k})+\left\langle \varLambda_{k},\,\nabla G\left(x_{k}\right)\right\rangle \right\Vert ^{2}+\left\Vert \varepsilon_{k}\right\Vert ^{2}.
\]

From the definition of $\tilde{x}_{k+1}$ in (\ref{hatx}), and by
the same arguments used to derive (\ref{telescope}), we have that
for any $x\in\mathcal{X}$,

\begin{equation}
\begin{array}{rcl}
 &  & \left\langle \tilde{x}_{k}-x,\,-\varepsilon_{k}\right\rangle \leq\frac{1}{2\gamma}\left(\left\Vert x-\tilde{x}_{k}\right\Vert ^{2}-\left\Vert x-\tilde{x}_{k+1}\right\Vert ^{2}\right)+\frac{\gamma}{2}\left\Vert \varepsilon_{k}\right\Vert ^{2}.\end{array}\label{teledifference}
\end{equation}
Notice that

\[
\begin{array}{rcl}
 &  & \left\langle x_{k}-x,\,\nabla f(x_{k})+\left\langle \varLambda_{k},\,\nabla G\left(x_{k}\right)\right\rangle +\varepsilon_{k}\right\rangle +\left\langle \tilde{x}_{k}-x,\,-\varepsilon_{k}\right\rangle \\
 & = & \left\langle x_{k}-x,\,\nabla f(x_{k})+\left\langle \varLambda_{k},\,\nabla G\left(x_{k}\right)\right\rangle \right\rangle +\left\langle x_{k}-\tilde{x}_{k},\,\varepsilon_{k}\right\rangle ,
\end{array}
\]
and $\left|\left\langle x_{k}-\tilde{x}_{k},\,\varepsilon_{k}\right\rangle \right|\leq\sqrt{D_{\mathcal{X}}}\left\Vert \varepsilon_{k}\right\Vert $,
thus by summing (\ref{telescope}) and (\ref{teledifference}) we
arrive at the conclusion.
\end{proof}
We may now analyze the averaged primal and dual solutions $\overline{x}_{K}=\frac{1}{K}\sum_{k=0}^{K-1}x_{k}$
and $\overline{\varLambda}_{K}=\frac{1}{K}\sum_{k=0}^{K-1}\varLambda_{k}$
returned by our primal-dual algorithm by giving a bound on the difference
$L_{\kappa}(\overline{x}_{K},\,\varLambda)-L_{\kappa}(x,\,\overline{\varLambda}_{K})$
for any $x\in\mathcal{X}$ and $\varLambda\in\mathcal{M}^{cd}\left(\Xi\right)\cap\mathcal{V}$.
This result is the key part of the proof of our main result Theorem
\ref{PSMDthm}. 
\begin{lem}
\label{twosides} For any $x\in\mathcal{X}$ and $\varLambda\in\mathcal{M}^{cd}\left(\Xi\right)\cap\mathcal{V}$,
we have,

\begin{equation}
\begin{array}{rcl}
L_{\kappa}(\overline{x}_{K},\,\varLambda)-L_{\kappa}(x,\,\overline{\varLambda}_{K}) & \leq & \frac{1}{\gamma K}B(\varLambda,\,\varLambda_{u})+\frac{1}{\gamma K}D_{\mathcal{X}}+\frac{\gamma}{2}(\bar{\rho}(\theta)G_{\max}^{2}+2(L_{f}+\bar{\rho}(\theta)L_{g,\,\mathcal{X}})^{2})\\
 &  & +\frac{3}{2}\gamma\,\frac{1}{K}\sum_{k=0}^{K-1}\left\Vert \varepsilon_{k}\right\Vert ^{2}+\sqrt{D_{\mathcal{X}}}\,\frac{1}{K}\sum_{k=0}^{K-1}\left\Vert \varepsilon_{k}\right\Vert .
\end{array}\label{violation}
\end{equation}
\end{lem}

\begin{proof}
From the convexity of $f$ and $G$, 

\[
\begin{array}{rcl}
 &  & L_{\kappa}(x_{k},\,\varLambda)-L_{\kappa}(x,\,\varLambda_{k})\\
 & = & f(x_{k})+\left\langle \varLambda,\,G(x_{k})\right\rangle -\kappa\,B\left(\varLambda,\,\varLambda_{u}\right)-f(x)-\left\langle \varLambda_{k},\,G(x)\right\rangle +\kappa\,B\left(\varLambda_{k},\,\varLambda_{u}\right)\\
 & = & f(x_{k})-f(x)+\left\langle \varLambda_{k},\,G(x_{k})-G(x)\right\rangle +\left\langle \varLambda-\varLambda_{k},\,G(x_{k})\right\rangle -\kappa\,(B\left(\varLambda,\,\varLambda_{u}\right)-B\left(\varLambda_{k},\,\varLambda_{u}\right))\\
 & \leq & \left\langle x_{k}-x,\,\nabla f(x_{k})+\left\langle \varLambda_{k},\,\nabla G\left(x_{k}\right)\right\rangle \right\rangle +\left\langle \varLambda-\varLambda_{k},\,G(x_{k})\right\rangle -\kappa\,(B\left(\varLambda,\,\varLambda_{u}\right)-B\left(\varLambda_{k},\,\varLambda_{u}\right)).
\end{array}
\]
By using Lemma \ref{mu-mut} and Lemma \ref{xt-x}, and adding these
inequalities together for all $k=0,1,\ldots,K-1$, we obtain

\[
\begin{array}{rcl}
 &  & \sum_{k=0}^{K-1}L_{\kappa}(x_{k},\,\varLambda)-L_{\kappa}(x,\,\varLambda_{k})\\
 & \leq & \kappa(B(\varLambda_{0},\,\varLambda_{u})-B(\varLambda_{k},\,\varLambda_{u}))+\frac{1}{\gamma}(B(\varLambda,\,\varLambda_{u})-B(\varLambda,\,\varLambda_{k}))+\frac{1}{2\gamma}(\left\Vert x-x_{0}\right\Vert ^{2}-\left\Vert x-x_{k}\right\Vert ^{2})\\
 &  & +\frac{1}{2\gamma}(\left\Vert x-\tilde{x}_{0}\right\Vert ^{2}-\left\Vert x-\tilde{x}_{K}\right\Vert ^{2})+\frac{\gamma K}{2}(\bar{\rho}(\theta)G_{\max}^{2}+2(L_{f}+\bar{\rho}(\theta)L_{g,\,\mathcal{X}})^{2})\\
 &  & +\frac{3}{2}\gamma\sum_{k=0}^{K-1}\left\Vert \varepsilon_{k}\right\Vert ^{2}+\sqrt{D_{\mathcal{X}}}\sum_{k=0}^{K-1}\left\Vert \varepsilon_{k}\right\Vert \\
 & \leq & \frac{1}{\gamma}B(\varLambda,\,\varLambda_{u})+\frac{1}{\gamma}\left\Vert x-x_{0}\right\Vert ^{2}+\frac{\gamma K}{2}(\bar{\rho}(\theta)G_{\max}^{2}+2(L_{f}+\bar{\rho}(\theta)L_{g,\,\mathcal{X}})^{2})+\frac{3}{2}\gamma\sum_{k=0}^{K-1}\left\Vert \varepsilon_{k}\right\Vert ^{2}+\sqrt{D_{\mathcal{X}}}\sum_{k=0}^{K-1}\left\Vert \varepsilon_{k}\right\Vert \\
 & \leq & \frac{1}{\gamma}B(\varLambda,\,\varLambda_{u})+\frac{1}{\gamma}D_{\mathcal{X}}+\frac{\gamma K}{2}(\bar{\rho}(\theta)G_{\max}^{2}+2(L_{f}+\bar{\rho}(\theta)L_{g,\,\mathcal{X}})^{2})+\frac{3}{2}\gamma\sum_{k=0}^{K-1}\left\Vert \varepsilon_{k}\right\Vert ^{2}+\sqrt{D_{\mathcal{X}}}\sum_{k=0}^{K-1}\left\Vert \varepsilon_{k}\right\Vert ,
\end{array}
\]
where the second inequality holds because the non-positive terms are
dropped and $\varLambda_{0}=\varLambda_{u}$, and the last inequality
follows from $\left\Vert x-x_{0}\right\Vert ^{2}\leq D_{\mathcal{X}}$. 

From the definitions of $\overline{x}_{K}$ and $\overline{\varLambda}_{K}$,
the convexity of $f$ and $G$, and the convexity of $B\left(\varLambda,\,\varLambda_{u}\right)$
in $\varLambda$, we have by Jensen's inequality

\[
\begin{array}{rcl}
L_{\kappa}(\overline{x}_{K},\,\varLambda)-L_{\kappa}(x,\,\overline{\varLambda}_{K}) & \leq & \frac{1}{K}[\sum_{k=0}^{K-1}L_{\kappa}(x_{k},\,\varLambda)-L_{\kappa}(x,\,\varLambda_{k})]\\
 & \leq & \frac{1}{\gamma K}B(\varLambda,\,\varLambda_{u})+\frac{1}{\gamma K}D_{\mathcal{X}}+\frac{\gamma}{2}(\bar{\rho}(\theta)G_{\max}^{2}+2(L_{f}+\bar{\rho}(\theta)L_{g,\,\mathcal{X}})^{2})\\
 &  & +\frac{3}{2}\gamma\,\frac{1}{K}\sum_{k=0}^{K-1}\left\Vert \varepsilon_{k}\right\Vert ^{2}+\sqrt{D_{\mathcal{X}}}\,\frac{1}{K}\sum_{k=0}^{K-1}\left\Vert \varepsilon_{k}\right\Vert .
\end{array}
\]
\end{proof}

\subsection{\label{subsec:Sample-Complexity-ofMC_integration}Sample Complexity
of Monte Carlo Integration}

Recall that $R(r)=(1+\beta)(\frac{1+r}{1-r})^{l}-1$ is defined in
Section \ref{sec:MC_integration}, we provide some properties of this
function which will be used in our analysis. 
\begin{lem}
\label{increasingproperty}For $l\in[\frac{1}{2},1)$, $\beta>0$
and $r\in(0,1)$, we have $R(r)>r$. Furthermore, for any $\varepsilon>0$
and positive integer $K$, there exists $\eta(\varepsilon,K)\in(0,\varepsilon)$,
such that as long as $\beta\leq\eta(\varepsilon,K)$, we have $R^{k}(\beta)\leq\varepsilon$
for all $k=1,2,\ldots,K-1$.
\end{lem}

\begin{proof}
Define a function $F(r)=(1+r)^{1-l}(1-r)^{l}$. It is sufficient to
show $F(r)<1+\beta$ for any $r\in(0,1)$. Noticing $F(0)=1<1+\beta$,
we only need to show $F(r)$ is decreasing in $r\in(0,1)$. In fact,
$\ln F(r)=(1-l)\ln(1+r)+l\ln(1-r)$, and $\frac{d\ln F(r)}{dr}=\frac{1-2l-r}{1-r^{2}}<0$. 

Since $R(r)$ is continuous and increasing in $r$, $\lim_{\beta\downarrow0}R(\beta)=0$,
and $R(r)>r$, there exists $\eta(\varepsilon,K)\in(0,\varepsilon)$,
such that as long as $\beta\leq\eta(\varepsilon,K)$, we have $R^{k}(\beta)\leq\varepsilon$
for all $k=1,2,\ldots,K-1$.
\end{proof}
The next lemma quantifies the error propagation from $\max_{1\leq i\leq N}\left|\frac{\hat{\lambda}_{k}(\xi_{i})}{\lambda_{k}(\xi_{i})}-1\right|$
to $\max_{1\leq i\leq N}\left|\frac{\hat{\lambda}_{k+1}(\xi_{i})}{\lambda_{k+1}(\xi_{i})}-1\right|$. 
\begin{lem}
\label{error_propagation_lemma}For $k\geq1$, suppose $\left|\frac{\intop_{\Xi}s_{k}(\xi)\lambda_{k}(\xi)^{l}/\text{vol}\left(\Xi\right)d\xi}{\sum_{i=1}^{N}s_{k}(\xi_{i})\lambda_{k}(\xi_{i})^{l}/N}-1\right|\leq\beta\in(0,1)$
and $\max_{i=1,\,2,\ldots,\,N}\left|\frac{\hat{\lambda}_{k}(\xi_{i})}{\lambda_{k}(\xi_{i})}-1\right|\leq r\in(0,1)$,
then $\max_{i=1,\,2,\ldots,\,N}\left|\frac{\hat{\lambda}_{k+1}(\xi_{i})}{\lambda_{k+1}(\xi_{i})}-1\right|\leq(1+\beta)\frac{(1+r)^{l}}{(1-r)^{l}}-1$. 
\end{lem}

\begin{proof}
From $\frac{\hat{\lambda}_{k+1}(\xi_{i})}{\lambda_{k+1}(\xi_{i})}=\frac{\min\left\{ \rho/\left(\frac{\text{vol}\left(\Xi\right)}{N}\sum_{i=1}^{N}s_{k}(\xi_{i})\hat{\lambda}_{k}(\xi_{i})^{l}\right),1\right\} }{\min\left\{ \rho/\left(\intop_{\Xi}s_{k}(\xi)\lambda_{k}(\xi)^{l}d\xi\right),1\right\} }\frac{\hat{\lambda}_{k}(\xi_{i})^{l}}{\lambda_{k}(\xi_{i})^{l}}$,
we have

\begin{equation}
\min\left\{ \frac{\intop_{\Xi}s_{k}(\xi)\lambda_{k}(\xi)^{l}/\text{vol}\left(\Xi\right)d\xi}{\frac{1}{N}\sum_{i=1}^{N}s_{k}(\xi_{i})\hat{\lambda}_{k}(\xi_{i})^{l}},1\right\} \left(\frac{\hat{\lambda}_{k}(\xi_{i})}{\lambda_{k}(\xi_{i})}\right)^{l}\leq\frac{\hat{\lambda}_{k+1}(\xi_{i})}{\lambda_{k+1}(\xi_{i})}\leq\max\left\{ \frac{\intop_{\Xi}s_{k}(\xi)\lambda_{k}(\xi)^{l}/\text{vol}\left(\Xi\right)d\xi}{\frac{1}{N}\sum_{i=1}^{N}s_{k}(\xi_{i})\hat{\lambda}_{k}(\xi_{i})^{l}},1\right\} \left(\frac{\hat{\lambda}_{k}(\xi_{i})}{\lambda_{k}(\xi_{i})}\right)^{l},\label{k+1_errorinlambda_estimation}
\end{equation}
where we use that for any $a,b,\rho>0$, $\min\left\{ \frac{b}{a},1\right\} \leq\frac{\min\left(\rho/a,1\right)}{\min\left(\rho/b,1\right)}\leq\max\left\{ \frac{b}{a},1\right\} $.
From $\left|\frac{\intop_{\Xi}s_{k}(\xi)\lambda_{k}(\xi)^{l}/\text{vol}\left(\Xi\right)d\xi}{\sum_{i=1}^{N}s_{k}(\xi_{i})\lambda_{k}(\xi_{i})^{l}/N}-1\right|\leq\beta$,
we have 

\begin{equation}
1-\beta\leq\frac{\intop_{\Xi}s_{k}(\xi)\lambda_{k}(\xi)^{l}/\text{vol}\left(\Xi\right)d\xi}{\sum_{i=1}^{N}s_{k}(\xi_{i})\lambda_{k}(\xi_{i})^{l}/N}\leq1+\beta.\label{first_ratio}
\end{equation}
From $\max_{i=1,\,2,\ldots,\,N}\left|\frac{\hat{\lambda}_{k}(\xi_{i})}{\lambda_{k}(\xi_{i})}-1\right|\leq r$,
we have 

\begin{equation}
\frac{1}{(1+r)^{l}}\leq\frac{\sum_{i=1}^{N}s_{k}(\xi_{i})\lambda_{k}(\xi_{i})^{l}/N}{\sum_{i=1}^{N}s_{k}(\xi_{i})\hat{\lambda}_{k}(\xi_{i})^{l}/N}\leq\frac{1}{(1-r)^{l}},\quad i=1,\,2,\ldots,\,N.\label{second_ratio}
\end{equation}

Noticing $\frac{\intop_{\Xi}s_{k}(\xi)\lambda_{k}(\xi)^{l}/\text{vol}\left(\Xi\right)d\xi}{\frac{1}{N}\sum_{i=1}^{N}s_{k}(\xi_{i})\hat{\lambda}_{k}(\xi_{i})^{l}}=\frac{\intop_{\Xi}s_{k}(\xi)\lambda_{k}(\xi)^{l}/\text{vol}\left(\Xi\right)d\xi}{\sum_{i=1}^{N}s_{k}(\xi_{i})\lambda_{k}(\xi_{i})^{l}/N}\frac{\sum_{i=1}^{N}s_{k}(\xi_{i})\lambda_{k}(\xi_{i})^{l}/N}{\sum_{i=1}^{N}s_{k}(\xi_{i})\hat{\lambda}_{k}(\xi_{i})^{l}/N}$,
we have from (\ref{k+1_errorinlambda_estimation}), (\ref{first_ratio})
and (\ref{second_ratio}), 

\begin{equation}
(1-\beta)\frac{(1-r)^{l}}{(1+r)^{l}}\leq\frac{\hat{\lambda}_{k+1}(\xi_{i})}{\lambda_{k+1}(\xi_{i})}\leq(1+\beta)\frac{(1+r)^{l}}{(1-r)^{l}},\quad i=1,\,2,\ldots,\,N,\label{errorpropagate}
\end{equation}
which implies 

\[
\max_{i=1,\,2,\ldots,\,N}\left|\frac{\hat{\lambda}_{k+1}(\xi_{i})}{\lambda_{k+1}(\xi_{i})}-1\right|\leq\max\left\{ (1+\beta)\frac{(1+r)^{l}}{(1-r)^{l}}-1,1-(1-\beta)\frac{(1-r)^{l}}{(1+r)^{l}}\right\} =(1+\beta)\frac{(1+r)^{l}}{(1-r)^{l}}-1,
\]
following from $(1+\beta)\frac{(1+r)^{l}}{(1-r)^{l}}-1>1-(1-\beta)\frac{(1-r)^{l}}{(1+r)^{l}}$. 

The next result paves the way for the error bounds in Theorem \ref{errorbounds_MonteCarlo}
of our primal-dual algorithm based on Monte Carlo integration. 
\end{proof}
\begin{lem}
\label{sample_complexity_lemma} For any $\varepsilon>0$, $\delta\in(0,1)$,
$K\geq\frac{2(C(\epsilon,\,\theta)+D_{\mathcal{X}})\bar{\kappa}(\epsilon)^{2}}{\bar{\rho}(\theta)G_{\max}^{2}+2(L_{f}+\bar{\rho}(\theta)L_{g,\,\mathcal{X}})^{2}}$,
and sample size

\[
N\geq\max\left\{ 2\frac{L_{g,\mathcal{X}}^{2}M^{2}\rho_{0}^{2}}{\varepsilon^{2}}\ln(\frac{4K}{\delta}),\frac{1}{2\eta(\frac{\varepsilon}{2\rho_{0}L_{g,\mathcal{X}}M},K)^{2}}\ln(\frac{4K}{\delta})\right\} ,
\]
we have

\[
\left|\frac{1}{N}\sum_{i=1}^{N}\lambda_{k}(\xi_{i})\nabla_{x}g\left(x_{k},\,\xi_{i}\right)-\int_{\Xi}\lambda_{k}(\xi)\nabla_{x}g\left(x_{k},\,\xi\right)/\text{vol}\left(\Xi\right)d\xi\right|\leq\frac{\varepsilon}{2\rho_{0}},\quad k=0,1,\ldots,K-1,
\]

\[
\max_{i=1,\,2,\ldots,\,N}\left|\frac{\hat{\lambda}_{k}(\xi_{i})}{\lambda_{k}(\xi_{i})}-1\right|\leq\frac{\varepsilon}{2\rho_{0}L_{g,\mathcal{X}}M},\quad k=0,1,\ldots,K-1,
\]
with probability at least $1-\delta$. 
\end{lem}

\begin{proof}
First, noticing that $\left|\lambda_{k}(\xi)\nabla_{x}g\left(x_{k},\,\xi\right)\right|\leq L_{g,\mathcal{X}}M$
for all $0\leq k\leq K-1$ and $\xi\in\Xi$, and the number of samples
$N\geq2\frac{L_{g,\mathcal{X}}^{2}M^{2}\rho_{0}^{2}}{\varepsilon^{2}}\ln(\frac{4K}{\delta})$,
from Hoeffding's inequality, we have for each $k=0,1,\ldots,K-1$,

\begin{equation}
\left|\frac{1}{N}\sum_{i=1}^{N}\lambda_{k}(\xi_{i})\nabla_{x}g\left(x_{k},\,\xi_{i}\right)-\int_{\Xi}\lambda_{k}(\xi)\nabla_{x}g\left(x_{k},\,\xi\right)/\text{vol}\left(\Xi\right)d\xi\right|\leq\frac{\varepsilon}{2\rho_{0}}\label{hoeffdingforMC}
\end{equation}
with probability at least $1-\frac{\delta}{2K}$.

Next, since the number of samples satisfies $N\geq\frac{1}{2\eta(\frac{\varepsilon}{2\rho_{0}L_{g,\mathcal{X}}M},K)^{2}}\ln(\frac{4K}{\delta})$,
from Hoeffding's inequality, we have for each $k=0,1,\ldots,K-1$,

\begin{equation}
\left|\frac{\intop_{\Xi}s_{k}(\xi)\lambda_{k}(\xi)^{l}/\text{vol}\left(\Xi\right)d\xi}{\sum_{i=1}^{N}s_{k}(\xi_{i})\lambda_{k}(\xi_{i})^{l}/N}-1\right|\leq\eta(\frac{\varepsilon}{2\rho_{0}L_{g,\mathcal{X}}M},K)\label{Hoeffdingforratio}
\end{equation}
with probability at least $1-\frac{\delta}{2K}$. 

For any $i\in\left\{ 1,\,2,\ldots,\,N\right\} $, $\frac{\widetilde{\lambda_{1}}(\xi_{i})}{\lambda_{1}(\xi_{i})}=\frac{\min\left\{ \rho/\left(\frac{\text{vol}\left(\Xi\right)}{N}\sum_{i=1}^{N}s_{0}(\xi_{i})\right),1\right\} }{\min\left\{ \rho/\left(\intop_{\Xi}s_{0}(\xi)d\xi\right),1\right\} }$,
using $\min\left\{ \frac{b}{a},1\right\} \leq\frac{\min\left(\rho/a,1\right)}{\min\left(\rho/b,1\right)}\leq\max\left\{ \frac{b}{a},1\right\} $
for any $a,b,\rho>0$, we have

\[
\min\left\{ \frac{\intop_{\Xi}s_{0}(\xi)d\xi}{\frac{\text{vol}\left(\Xi\right)}{N}\sum_{i=1}^{N}s_{0}(\xi_{i})},1\right\} \leq\frac{\widetilde{\lambda_{1}}(\xi_{i})}{\lambda_{1}(\xi_{i})}\leq\max\left\{ \frac{\intop_{\Xi}s_{0}(\xi)d\xi}{\frac{\text{vol}\left(\Xi\right)}{N}\sum_{i=1}^{N}s_{0}(\xi_{i})},1\right\} ,
\]
and then from (\ref{Hoeffdingforratio}) with $k=0$, 

\begin{equation}
\max_{i=1,\,2,\ldots,\,N}\left|\frac{\widetilde{\lambda_{1}}(\xi_{i})}{\lambda_{1}(\xi_{i})}-1\right|\leq\left|\frac{\intop_{\Xi}s_{0}(\xi)/\text{vol}\left(\Xi\right)d\xi}{\sum_{i=1}^{N}s_{0}(\xi_{i})/N}-1\right|\leq\eta(\frac{\varepsilon}{2\rho_{0}L_{g,\mathcal{X}}M},K)\label{firsterrorin_lambda_estimation}
\end{equation}
with probability at least $1-\frac{\delta}{2K}$. 

From $K\geq\frac{2(C(\epsilon,\,\theta)+D_{\mathcal{X}})\bar{\kappa}(\epsilon)^{2}}{\bar{\rho}(\theta)G_{\max}^{2}+2(L_{f}+\bar{\rho}(\theta)L_{g,\,\mathcal{X}})^{2}}$,
we have $l\in[\frac{1}{2},1)$. Invoking Lemmas \ref{increasingproperty}
and \ref{error_propagation_lemma}, from (\ref{Hoeffdingforratio})
and (\ref{firsterrorin_lambda_estimation}) and taking the union bound,
we have 

\[
\max_{i=1,\,2,\ldots,\,N}\left|\frac{\hat{\lambda}_{k}(\xi_{i})}{\lambda_{k}(\xi_{i})}-1\right|\leq\frac{\varepsilon}{2\rho_{0}L_{g,\mathcal{X}}M},\quad k=1,2,\ldots,K-1,
\]
with probability at least $1-\delta/2$. Together with (\ref{hoeffdingforMC})
and taking union bound again, we arrive at the conclusion. 
\end{proof}

\section{\label{sec:Numerical-Experiments}Numerical Experiments}

This section presents a case adapted from \cite{mehrotra2014cutting}
to illustrate the methods developed in this paper:
\begin{alignat}{1}
\min\  & \left(x_{1}-2\right)^{2}+\left(x_{2}-0.2\right)^{2}\label{test_problem}\\
\text{s.t. } & \left(\frac{5\sin\left(\pi\sqrt{t}\right)}{1+t^{2}}\right)x_{1}^{2}-x_{2}\le0,\ \ \forall t\in\left[0,1\right],\nonumber \\
 & x_{1}\in\left[-1,1\right],x_{2}\in\left[0,0.2\right].\nonumber 
\end{alignat}

We solve the above problem by Algorithm \ref{alg:The-Randomized-Primal-Dual}
with Monte Carlo integration. The optimal solution of Problem (\ref{test_problem})
is $x=\left(0.20523677,0.2\right)$, with the optimal value being
$3.221$ \cite{mehrotra2014cutting}. Table \ref{tab:Results_rule1}
reports the numerical results by changing the number of iterations
from $K=500$ to $K=60000$. We see that the objective value, i.e.
$f\left(\bar{x}_{K}\right)$, reaches $3.201$ after $60000$ iterations,
which is close to the optimal value. In addition, the constraint violations
of the algorithm decrease with $K$. Figure \ref{fig:result_1}(a)
and (b) show the changes of $f\left(\bar{x}_{K}\right)$ and $g\left(\bar{x}_{K}\right)$
with the number of iterations $K$, respectively.

Sensitivity analysis is also implemented on the number of samples
for Monte Carlo integration. In Figure \ref{fig:sen_analysis}, we
plot the values corresponding to $x_{k}$, instead of $\bar{x}_{K}$,
to see how the objective values and constraint violations change in
each iteration. As can be seen, when the number of samples increases,
the convergence to the optimal value is faster. In addition, the constraint
violation improves as we increase the number of samples. When $N$
is increased from $100$ to $1000$, the improvement is not significant,
suggesting that $N=100$ is a sufficiently large sample size for Problem
(\ref{test_problem}).

We also study the behavior of the last iterate $x_{K}$. We see that
$f\left(x_{K}\right)$ converges faster to the optimal value in comparison
to $f\left(\bar{x}_{K}\right)$; when $N=1000$, $f\left(x_{K}\right)$
is close to the optimal value in around 7000 iterations, but $f\left(\bar{x}_{K}\right)$
is only $3.196$ after $50000$ steps (see Figure \ref{fig:result_1}
and \ref{fig:sen_analysis}).

\begin{table}
\caption{\label{tab:Results_rule1}Simulation results with $N=1000$ and $\epsilon=0.001$.}

\centering{}%
\begin{tabular}{|c|c|c|c|c|c|c|c|c|c|c|c|}
\hline 
\multirow{2}{*}{} & \multicolumn{10}{c|}{Number of iterations} & \multirow{2}{*}{Benchmark}\tabularnewline
\cline{2-11} \cline{3-11} \cline{4-11} \cline{5-11} \cline{6-11} \cline{7-11} \cline{8-11} \cline{9-11} \cline{10-11} \cline{11-11} 
 & $500$ & $1000$ & $3000$ & $5000$ & $10000$ & $20000$ & $30000$ & $40000$ & $50000$ & $60000$ & \tabularnewline
\hline 
\hline 
$f\left(\bar{x}_{K}\right)$ & $2.907$ & $2.997$ & $3.088$ & $3.119$ & $3.153$ & $3.176$ & $3.185$ & $3.192$ & $3.196$ & $3.201$ & $3.221$\tabularnewline
\hline 
$G\left(\bar{x}_{K}\right)$ & $0.235$ & $0.161$ & $0.095$ & $0.074$ & $0.052$ & $0.038$ & $0.032$ & $0.028$ & $0.025$ & $0.023$ & $\le0$\tabularnewline
\hline 
\end{tabular}
\end{table}

\begin{figure}
\begin{centering}
\includegraphics[viewport=15bp 0bp 425bp 320bp,clip,scale=0.57]{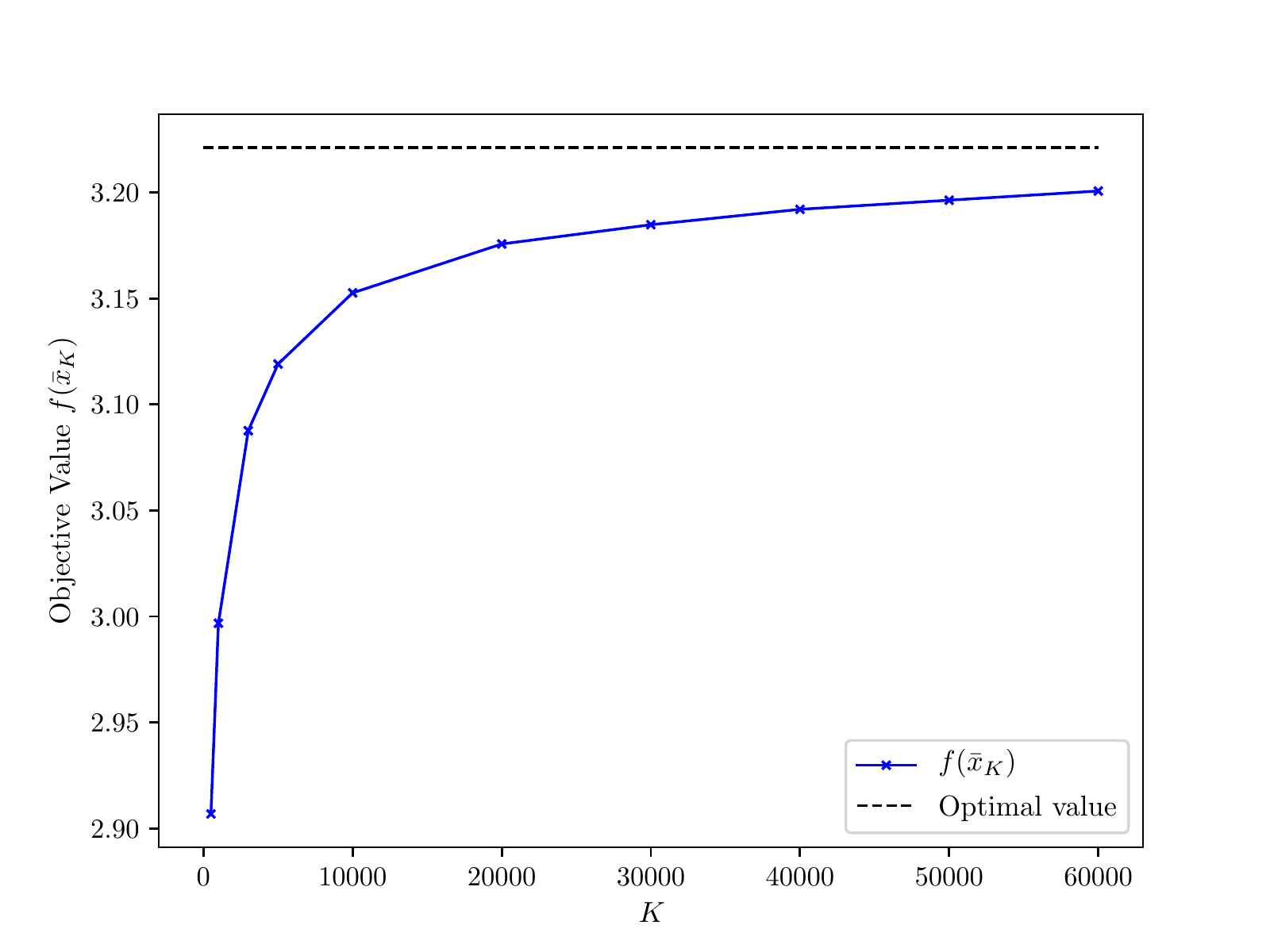}\includegraphics[viewport=15bp 0bp 425bp 320bp,clip,scale=0.57]{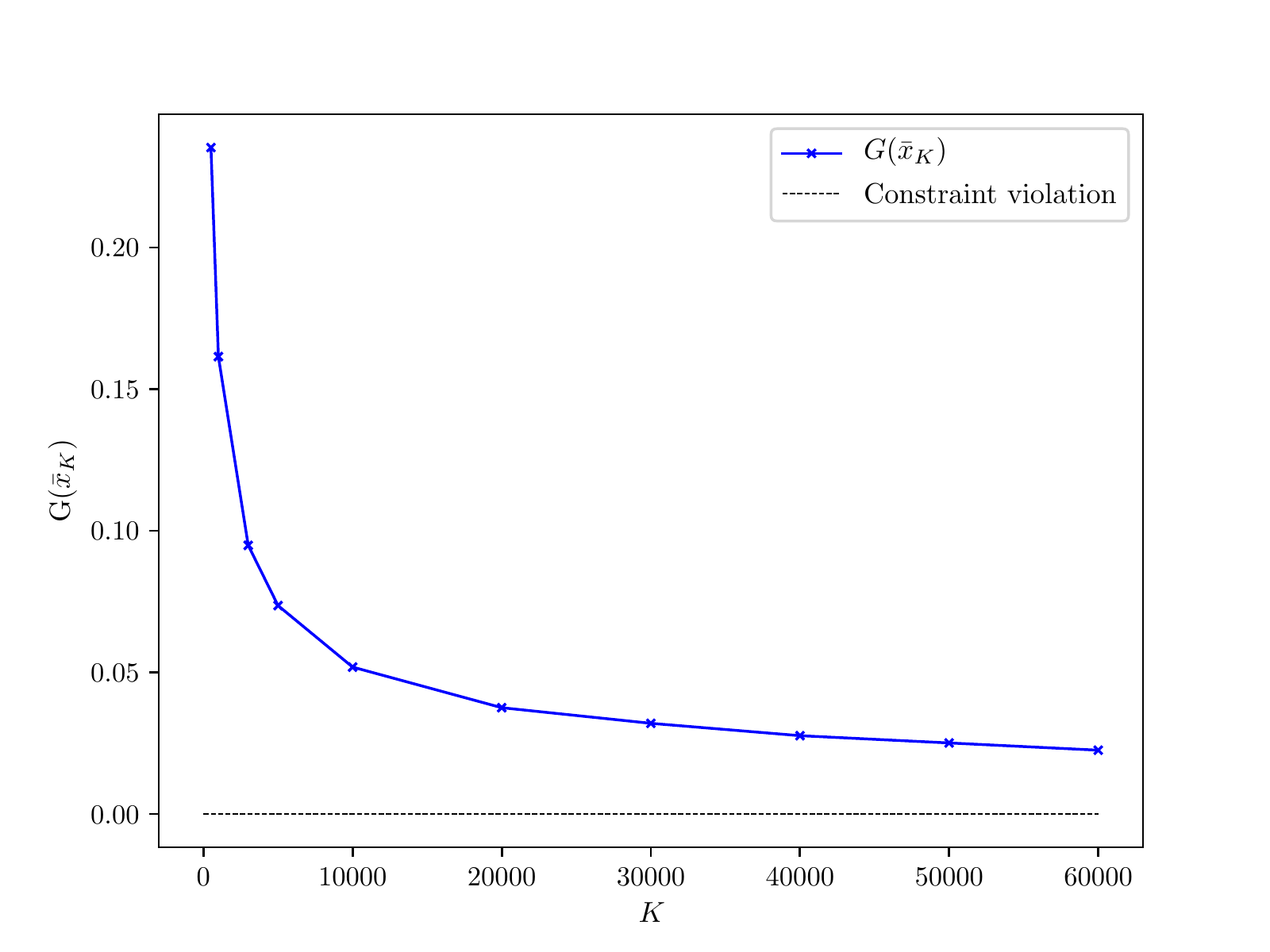}
\par\end{centering}
\caption{\label{fig:result_1}Sensitivity analysis on the number of iterations
$K$ results given $N=1000$: (a) objective values $f\left(\bar{x}_{K}\right)$
and (b) the constraint violations $G\left(\bar{x}_{K}\right)$.}
\end{figure}

\begin{figure}
\begin{centering}
\includegraphics[viewport=15bp 0bp 425bp 320bp,clip,scale=0.57]{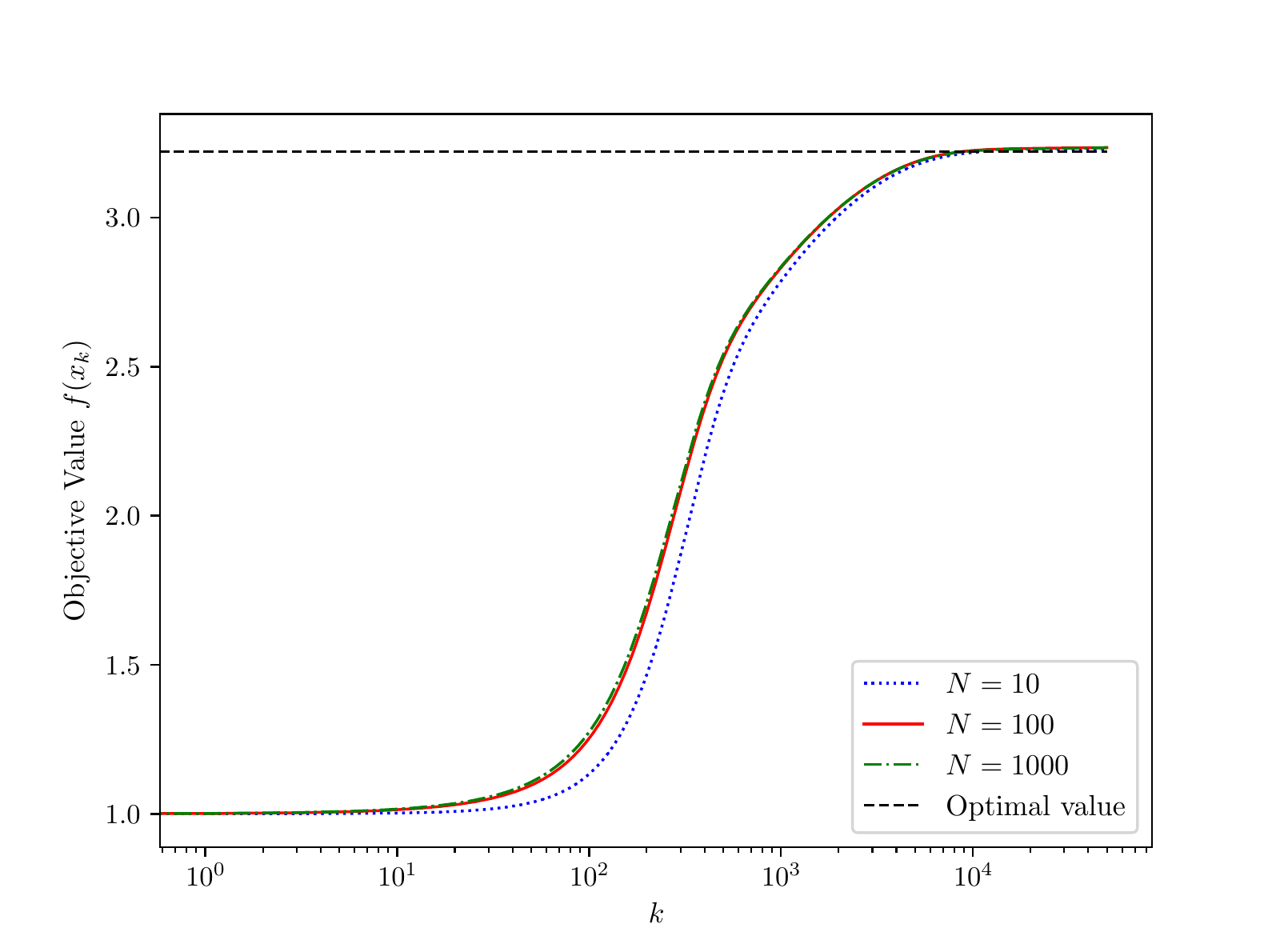}\includegraphics[viewport=15bp 0bp 425bp 320bp,clip,scale=0.57]{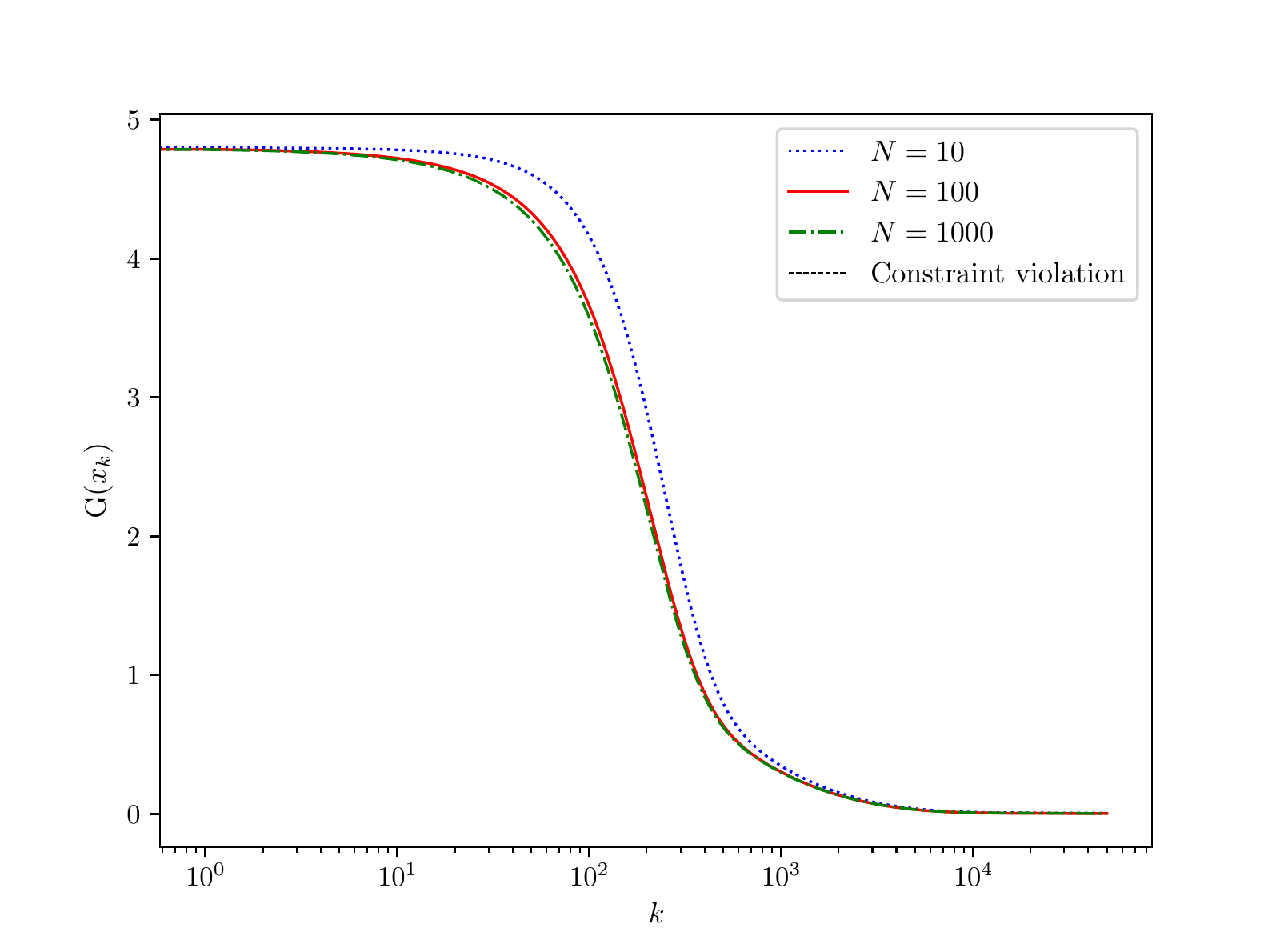}
\par\end{centering}
\caption{\label{fig:sen_analysis}Sensitivity analysis on the number of samples
by fixing $K=50000$.}
\end{figure}

\section{\label{sec:Conclusion}Conclusion}

In this paper, we develope an inexact primal-dual algorithm for SIP,
provide general error bounds, and analyze a specific primal-dual algorithm
based on Monte Carlo integration. At the core of our contribution,
we propose a new prox function for $\mathcal{M}_{+}\left(\Xi\right)$
which plays a central role in our algorithm and its implementation.

We briefly remark on the connection between our results and the most
closely related work. In \cite{mahdavi2012online}, a constrained
stochastic optimization problem is solved via a randomized primal-dual
algorithm which achieves an $\mathcal{O}(1/\sqrt{K})$ convergence
rate. This rate depends on the number of constraints, while our main
result is independent of the number of constraints. In \cite{lin2017revisiting},
a saddle-point reformulation and root finding approach is applied
to solve semi-infinite linear programming problems arising from Markov
decision processes (MDP). The resulting 'ALP-Secant' algorithm is
based on solving a sequence of saddle-point problems. In our work,
we solve a convex SIP via a single saddle-point problem. 

In future research, we will strengthen our results by considering
more sophisticated primal-dual algorithms with better convergence
rates and more advanced constraint sampling schemes. We will also
look beyond first-order methods to investigate inexact second-order
methods for SIP.

\bibliographystyle{plain}
\bibliography{References}

\end{document}